\documentclass[10pt]{amsart}
\usepackage{amsmath}
\usepackage{amssymb}
\usepackage{amsfonts}
\usepackage{amsthm}
\usepackage{amsxtra}
\usepackage{fontenc}
\usepackage{fullpage}
\usepackage{hyperref}

\usepackage{xcolor} 
\usepackage{cancel} 
\usepackage{bbm} 
\theoremstyle{plain}
\newtheorem{theorem}{Theorem}
\newtheorem{corollary}{Corollary}

\newtheorem{conjecture}{Conjecture}
\theoremstyle{definition}

\theoremstyle{remark}
\newtheorem{remark}{Remark}[section]

\newcommand{\half}{
        {\lower0.00ex\hbox{\raise.6ex\hbox{\the\scriptfont0 1}
                           \kern-.5em\slash\kern-.1em\lower.45ex
                                     \hbox{\the\scriptfont0 2}}}}
\newcommand{\quarter}{
        {\lower0.00ex\hbox{\raise.6ex\hbox{\the\scriptfont0 1}
                           \kern-.5em\slash\kern-.1em\lower.45ex
                                     \hbox{\the\scriptfont0 4}}}}
\newcommand{\tquarter}{
        {\lower0.00ex\hbox{\raise.6ex\hbox{\the\scriptfont0 3}
                           \kern-.5em\slash\kern-.1em\lower.45ex
                                     \hbox{\the\scriptfont0 4}}}}
\newcommand{\eighth}{
        {\lower0.00ex\hbox{\raise.6ex\hbox{\the\scriptfont0 1}
                           \kern-.5em\slash\kern-.1em\lower.45ex
                                     \hbox{\the\scriptfont0 8}}}}
\newcommand{\othird}{
        {\lower0.00ex\hbox{\raise.6ex\hbox{\the\scriptfont0 1}
                           \kern-.5em\slash\kern-.1em\lower.45ex
                                     \hbox{\the\scriptfont0 3}}}}

\DeclareMathOperator{\Tr}{Tr}

\begin{document}

\title{Moments of the Gaussian $\beta$ Ensembles and the large-$N$ expansion of the densities}


\author{N.S.~Witte and P.J.~Forrester}
\address{Department of Mathematics and Statistics,
University of Melbourne, Parkville, Victoria 3010, Australia}
\email{\tt N.Witte@ms.unimelb.edu.au; P.Forrester@ms.unimelb.edu.au}

\begin{abstract}
The loop equation formalism is used to compute the $1/N$ expansion of the resolvent for
the Gaussian $\beta$ ensemble up to and including the term at $O(N^{-6})$. This allows
the moments of the eigenvalue density to be computed up to and including the
12-th power and the smoothed density to be expanded up to and including the term at $O(N^{-6})$.
The latter contain non-integrable singularities at the endpoints of the support --- we show
how to nonetheless make sense of the average of a sufficiently smooth linear statistic.
At the special couplings $\beta = 1$, $2$ and $4$ there are characterisations
of both the resolvent and the moments which allows for the corresponding expansions to be extended,
in some recursive form at least, to arbitrary order. In this regard we give fifth order linear differential
equations for the density and resolvent at $\beta = 1$ and $4$, which complements the known
third order linear differential equations for these quantities at $\beta = 2$.
\end{abstract}

\maketitle


\section{Introduction}\label{Introduction}
\setcounter{equation}{0}

One of the most active topics in random matrix theory at present is the study of the
$\beta$-ensembles. Motivations come from varying viewpoints, including universality,
integrability, asymptotics and applications to matrix models and field theories.
In applications to  matrix models and field theories the average
\begin{equation}
  \left \langle \Tr G^{2k} \right \rangle_{G \in {\rm GUE}^{*}} ,\quad k\in \mathbb{Z}_{\geq 0} ,
\label{muA_GUEstar}
\end{equation}
features prominently. Here GUE$^{*}$ ($ \beta=2 $) denotes the set of $ N \times N $ GUE matrices each
multiplied by $1/\sqrt{2N}$, where this scaling is chosen so that the leading order support of the eigenvalues is $(-1,1)$. The large $N$ asymptotics
of \eqref{muA_GUEstar} is one of the earliest examples of a topological expansion in the theory of
matrix integrals \cite{BIPZ_1978, Zv_1997, DiF_2001, LZ_2004}. This expansion
counts maps on surfaces of definite genus, or equivalently the number $c(g;k)$ of pairings of vertices of
a regular $2k$-gon which corresponds to a surface of genus $g$. Thus purely combinatorial reasoning gives
\begin{equation}
\frac{1}{N} \left \langle \Tr G^{2k} \right \rangle_{{\rm GUE}^{*}}
 = \sum_{g=0}^{[k/2]} \frac{ c(g;k)}{N^{2g}} .
\label{TRG_combinatoric}
\end{equation}
It is furthermore the case that
the average $ \langle \Tr G^{2k} \rangle $ with respect to  GOE$^{*}$ ($ \beta=1 $) or
GSE$^{*}$ ($ \beta=4 $) (see (\ref{GbetaE}) below for the precise definitions) also admits a combinatorial interpretation,
with the corresponding coefficients again being related to certain maps \cite{MW_2003}.

A primary concern of the present work is the analysis and generation of the large $N$ expansion
of the average $\langle {\rm Tr} \, G^{2k} \rangle$, where the matrices $G$ belong to the
Gaussian $\beta$-ensemble G$\beta$E$^{*}$. For general $\beta > 0$ such matrices can be defined
as real symmetric matrices with independent entries (see e.g.~\cite{rmt_Fo}), but since
Tr$\, G^{2k} = \sum_{m=1}^N \lambda_m^{2k}$, for present purposes it suffices to specify
the eigenvalue probability density function. This is proportional to
\begin{equation}\label{GbetaE}
   \prod^{N}_{l=1} e^{-\beta N\lambda_l^2} \prod_{1\leq j<k \leq N}|\lambda_k-\lambda_j|^{\beta} ,
\end{equation}
As in the notation GUE$^{*}$, which corresponds to the case $\beta = 2$, the asterisk
in G$\beta$E$^{*}$ denotes the particular scaling of the eigenvalues that gives the leading order eigenvalue density
supported on $(-1,1)$. There has been recent interest in this $\beta$ generalized moment
from the viewpoint of topological expansions of matrix models \cite{BH_2009, BEMS_2011, BEMP-F_2012, Mo_2012}.

More general than the calculation of the moments is the problem of computing the
asymptotic smoothed signed densities.
Consider a linear statistic of the eigenvalue, which refers to a function of the form
$A = \sum_{j=1}^N a(\lambda_j)$.
Let $ \rho^{N}_{(1)}(\lambda;\beta) $ denote the eigenvalue
density, which has the defining property that
$ \int^{b}_{a} \rho^{N}_{(1)}(\lambda;\beta) d\lambda $ is equal to the expected
number of eigenvalues in the interval $ [a,b] $. The mean $ \mu_{N}[A] $ of the linear statistic
$ A $ is given in terms of the eigenvalue density according to
\begin{equation}\label{Aa}
\mu_{N,\beta}[A] = \int_{-\infty}^\infty a(\lambda)  \rho^{N}_{(1)}(\lambda;\beta) \, d \lambda.
\end{equation}
Performing a large $N$ expansion analogous to (\ref{TRG_combinatoric}) allows the
asymptotic smoothed signed densities\\ \hfil
$\{ \tilde{\rho}_{(1),g}(\lambda;\beta)\}_{g=0,1,\dots} $ to be defined according to
\begin{equation}
   \frac{1}{N}\mu_{N,\beta}[A] = \sum^{\infty}_{g=0} \frac{1}{N^g} \int^{\infty}_{-\infty}a(\lambda)
   \tilde{\rho}_{(1),g}(\lambda;\beta) d\lambda .
\label{largeNmu}
\end{equation}
\noindent
In the case of $ \beta=2 $ \eqref{largeNmu} is known rigorously from the work of Ercolani and
McLaughlin \cite{EM_2003} and from Haagerup and Thorbj{\o}rnsen \cite{HT_2012},
and in fact all the terms with $ g $ odd vanish. For general $ \beta $
 a rigorous demonstration of \eqref{largeNmu} can be found
in the recent work \cite{BG_2013}. Our determination of $ \{ \tilde{\rho}_{(1),g}(\lambda;\beta) \} $
assumes the existence of the large $ N $ expansion \eqref{largeNmu}.

The best known result relating to \eqref{largeNmu}, which requires certain technical assumptions
on $ a(\lambda) $, is the limit theorem (see e.g. \cite{PS_2011})
\begin{equation}
    \lim_{N\to \infty} \frac{1}{N} \mu_{N}[A] = \int^{\infty}_{-\infty} a(\lambda)
    \tilde{ \rho}_{(1),0}(\lambda;\beta) \, d\lambda ,
\end{equation}
where, with $ \chi_{\lambda \in J}=1 $ for $ \lambda\in J $ and $ \chi_{ \lambda \in J}=0 $ otherwise,
\begin{equation}
     \tilde{ \rho}_{(1),0}(\lambda;\beta)   = \frac{2}{\pi} \sqrt{1-\lambda^2} \chi_{\lambda \in (-1,1)},
\label{wignerSC}
\end{equation}
this being the celebrated Wigner semi-circle law.
Note that $ \tilde{ \rho}_{(1),0}(\lambda;\beta) $ is  not dependent on $\beta$.

Let us illustrate our method of determination of
$ \tilde{\rho}_{(1),g}(\lambda;\beta) $ by deriving \eqref{wignerSC}. In fact we will make use of
one of the standard approaches (see e.g. \cite{PS_2011}) which proceeds by considering the
particular linear statistic
\begin{equation} \label{1.7oa}
a(\lambda)=(z-\lambda)^{-1}.
\end{equation}
For the Gaussian $\beta$-ensemble, with
\begin{equation}
   R(z) \equiv \lim_{N\to \infty} \frac{1}{N} \mu_{N}\left[ \sum^{N}_{j=1} \frac{1}{z-\lambda_j} \right] ,
\end{equation}
it is possible to deduce that \cite{Jo_1998}
\begin{equation}
  R(z)^2-4z R(z)+4 = 0 ,
\end{equation}
which gives
\begin{equation}
   R(z) = 2\left( z-\sqrt{z^2-1} \right) .
\end{equation}
But from the definition of $ R(z) $ and $\tilde{ \rho}_{(1),0}(x;\beta) $
\begin{equation}
  \tilde{ \rho}_{(1),0}(x;\beta) = \frac{1}{2\pi i} \lim_{\epsilon \to 0^+}\left( R(x-i\epsilon)-R(x+i\epsilon) \right) ,
\label{Sinverse}
\end{equation}
and \eqref{wignerSC} follows.

In addition to knowledge of the explicit form of $ \tilde{ \rho}_{(1),0}(x;\beta) $ as given by  \eqref{wignerSC},
it is furthermore well known that \cite{Jo_1998, FFG_2006}
\begin{equation}
  \tilde{\rho}_{(1),1}(\lambda;\beta) = \left( \frac{1}{\beta}-\frac{1}{2} \right)\left\{ \tfrac{1}{2}\left[ \delta(\lambda-1)+\delta(\lambda+1) \right] - \frac{1}{\pi\sqrt{1-\lambda^2}} \right\} .
\label{1stcorrection}
\end{equation}
Note that the dependence on $ \beta $ is a linear polynomial in $ 1/\beta $ which vanishes at
$ \beta=2 $, and that this latter feature is consistent with the expansion \eqref{largeNmu} only
involving even inverse powers of $ N $ for $ \beta=2 $. The result \eqref{1stcorrection} follows from
the $ 1/N $ term in the expansion
\begin{equation}
   \frac{1}{N} \mu_{N}\left[ \sum^{N}_{j=1} \frac{1}{z-\lambda_j} \right] = R(z)
  + \frac{1}{N} \left( \frac{1}{\beta}-\frac{1}{2} \right)\left[ \frac{1}{\sqrt{z^2-1}}-\frac{z}{z^2-1} \right] + {\rm O}(N^{-2}) ,
\label{leadingNmu}
\end{equation}
and application of \eqref{Sinverse}. It might then seem that our task is to extend the expansion
\eqref{leadingNmu} to higher order. While this is essentially the case, there are some complications.
For example, the next term in the expansion \eqref{leadingNmu} is \cite{BMS_2011,MMPS_2012}
\begin{equation}
  \frac{1}{N^2}\left\{ \tfrac{1}{4}\left( 1-\frac{2}{\beta} \right)^2 \left[ \frac{z^2+\frac{1}{4}}{(z^2-1)^{5/2}}-\frac{z}{(z^2-1)^2} \right] + \tfrac{1}{16}\frac{1}{(z^2-1)^{5/2}} \right\} .
\label{2ndcorrection}
\end{equation}
Application of \eqref{Sinverse} then gives that $ \rho_{(1),2}(\lambda;\beta) $ contains a term
proportional to $ (1-\lambda^2)^{-5/2} $ and thus is not integrable at $ \lambda=\pm 1$. One of
our tasks then is to give meaning to the terms in \eqref{largeNmu} in the light of such singularities.

In the existing literature one can find the expansion \eqref{leadingNmu} extended up to and including
the term $ {\rm O}(N^{-4}) $ in \cite{BMS_2011}, and up to and including the term $ {\rm O}(N^{-6}) $ in \cite{MMPS_2012}.
However, examination of the two sets of results show that they disagree in the term $ {\rm O}(N^{-4}) $.
Thus we have no option but to go back to scratch and to derive the expansion \eqref{leadingNmu} and
its extension to higher orders for ourselves. As in \cite{BMS_2011,MMPS_2012} we use the loop equation method, the
details of which are given in \S \ref{LoopEquations}.

A corollary of knowledge of the expansion \eqref{leadingNmu} (extended to higher orders) is
knowledge of the expansion of \eqref{largeNmu} in the case of $ a(\lambda)=\lambda^{2p} $, $ p \in \mathbb{Z}_{\ge 0} $,
corresponding to the moments of the density. Thus
\begin{equation}\label{1.14}
 \mu_{N,\beta}\left[ \sum^{N}_{j=1} \frac{1}{z-\lambda_j} \right] =
   \sum^{\infty}_{p=0} \frac{1}{z^{1+2p}} \mu_{N,\beta}\left[ \sum^{N}_{j=1}\lambda_j^{2p} \right]  =
   \sum_{p=0}^\infty \frac{1}{z^{1+2p}} \Big\langle {\rm Tr} \, G^{2p} \Big\rangle_{{\rm G}\beta{\rm E}^*}.
\end{equation}
This should be interpreted as an asymptotic expansion as the sum need not converge.
Moreover, the moments have the crucial property that their large $ N $ expansion terminates,
\begin{equation}\label{1.15}
  \frac{1}{N} \Big\langle  {\rm Tr}\, G^{2p} \Big\rangle_{{\rm G}\beta{\rm E}^*} = \sum^{p}_{g=0} \frac{1}{N^g} \int^{\infty}_{-\infty} \lambda^{2p} \tilde{\rho}_{(1),g}(\lambda;\beta) d\lambda .
\end{equation}
This will play a crucial role in us giving meaning to the expansion \eqref{largeNmu} for general
$ a(\lambda) $, in the light of the non-integrable singularities of the $ \tilde{\rho}_{(1),g}(\lambda;\beta) $
identified below \eqref{2ndcorrection}.

One possible approach to study the expansion \eqref{largeNmu} is to compute the large $ N $
expansion of the $ \rho_{(1)}(\lambda;\beta) $ itself, rather than (\ref{1.15}).
Actually this is rather complicated as there
are three distinct scaling regimes: $ -1< \lambda <1 $, $ \lambda \approx \pm 1+X/2N^{2/3} $, and
$ |\lambda|>1 $ which correspond to the bulk, soft edge and exponentially small portion of
$ \rho_{(1)}(\lambda;\beta) $ respectively. Moreover in the bulk regime there are both oscillatory and
non-oscillatory terms. In fact the explicit carrying out of this expansion  \cite{GFF_2005} up to including the
term at order $ N^{-2} $ shows that $ \tilde{\rho}_{(1),g}(\lambda;\beta) $ for
$g=0,1,2$ is the same as that obtained by expanding
$ \rho_{(1)}(\lambda;\beta) $ in the region $ -1< \lambda <1 $ and ignoring the oscillatory terms.
For $\beta = 2$ this expansion can be generated to higher order using the linear differential
equation satisfied by $ \rho_{(1)}$, see (\ref{GUE_densityODE}) below, and the same property is observed.

We commented that the original  motivation for considering the expansion (\ref{1.15}) came
from its relevance to matrix models and their topological interpretation. Further interest
was then identified with respect to the underlying asymptotic smoothed signed densities,
and in particular to the analytic challenge issued by their in general non-integrable
singularities at $x = \pm 1$. As suggested by the final sentence of the above paragraph,
integrability provides yet another motivation. It turns out that orthogonal polynomial expressions
for the density $\rho_{(1)}^N(\lambda;\beta)$ in the cases $\beta = 1,2$ and 4 (see
e.g.~\cite[Ch.~5\&6]{rmt_Fo}), can be used to determine linear differential equations for the
density and its Stieltjes transform or equivalently the resolvent. In the case $\beta = 2$ these differential equations, which
are both  third order with the same homogeneous part, are known from earlier work \cite{GT_2005, HT_2012}.

We begin in \S 2 by detailing the loop equation formalism as it applies to the
Gaussian $\beta$ ensemble, and we give the explicit form of the first three terms in the large
$N$ expansion. This is supplemented in \S 3 by the specification of the next
three terms in this expansion. This knowledge is used to compute the asymptotic smoothed
signed densities and moments up the corresponding order, and we furthermore address
the problem imposed by the singularities of the former in the computation of the integrals
in (\ref{largeNmu}). In \S 4 we first make note of known characterisations of the moments
for $\beta = 1,2$ and 4 to general order, as these provide checks on our results for general $\beta$.
We consider the problem of determining the resolvent at these couplings for general order,
and this leads us to the consideration of linear differential equations. In addition to giving a
self contained derivation of the known third order linear differential equations for the
density and resolvent at $\beta = 2$, we derive fifth order linear differential equations for
these quantities in the cases $\beta = 1$ and 4.

\section{Loop equations and the large $N$ expansion of the resolvent for the Gaussian $\beta$-ensemble}\label{LoopEquations}
\setcounter{equation}{0}

The Gaussian $\beta$-ensemble eigenvalue probability density function PDF \eqref{GbetaE} is the special case 
$\kappa = \beta/2$, $V(\lambda) = \frac{1}{2}\lambda^{2}$ of the eigenvalue PDF  proportional to
\begin{equation}
\prod_{l=1}^{N} e^{-N\kappa V(\lambda_{l})/g } \prod_{1 \le j < k \le N}|\lambda_{k} - \lambda_{j}|^{2\kappa} .
\label{pdfSTAR}
\end{equation}
We will assume henceforth that $ g, N, \kappa > 0 $.
Note that this is a slight variant on the average given in \eqref{GbetaE} where we have
introduced an extra coupling constant factor $ 4g $. This new average will be denoted as G$\beta$E$^{*}(g)$.
Much of the theory that follows is applicable to general weights which are parameterised in the form
\begin{equation}
  V = g_{0} + \sum_{k=1}^K  \frac{g_{k}} {k}\lambda^{k} .
\label{Vpotential}
\end{equation}
By definition the smoothed eigenvalue density is computed from knowledge of the mean (\ref{Aa}) of a sufficiently general
linear statistic $A$, with a particularly convenient choice being that
corresponding to (\ref{1.7oa}). Thus one wants to compute the so-called resolvent
\begin{equation}
W_{1}(x) := \left \langle \sum_{j=1}^{N} \frac{1}{x - \lambda_{j} } \right \rangle_{{\rm G{\beta}E^{*}(g)}}
         \mathop{\sim}_{x \rightarrow \infty} \sum_{k=0}^{\infty} \frac{m^{*}_{k}}{x^{k+1}},
  \quad m^{*}_{k} = \int_{-\infty}^{\infty} \lambda^{k} \rho_{(1)}^{N} (\lambda) d\lambda .
\label{resolvent}
\end{equation}
As already noted, interest in this quantity also stems from its relationship (\ref{1.14}), (\ref{1.15}) to the
moments of the eigenvalue density.

The large $N$ expansion of the resolvent is in fact a classical problem in random matrix theory
and the theory of matrix models. Its solution involves a recursive set of equations, known as either the
Pastur \cite{Pa_1972} or loop equations \cite{AM_1990} (we will use the latter terminology) in the mathematical literature, or as
Virasoro constraints, Schwinger-Dyson equations or Ward identities in the physics literature.

Several auxiliary quantities are required. Thus we introduce the correlators
\begin{gather}
   W_{n}(x_{1},\ldots,x_{n}) := \left\langle \sum_{i_{1}=1}^{N}\frac{1}{x_{1}-\lambda_{i_{1}} } \cdots \sum_{i_{n}=1}^{N}  \frac{1}{x_{n}-\lambda_{i_{n}} } \right\rangle_{c} ,
\notag \\
   P_{1}(x) := \left\langle \sum_{j=1}^{N}\frac{V^{\prime}(x)-V^{\prime}(\lambda_{j})}{x-\lambda_{j}} \right\rangle ,
\notag \\
\begin{aligned}
   P_{n+1}(x;x_{1},\ldots,x_{n}) &:=
   \left \langle \sum_{i=1}^{N}  \frac{V^{\prime}(x) - V^{\prime}(\lambda_{i})} {x- \lambda_{i} }
   \times \sum_{i_{1}=1}^{N} \frac{1}{x_{1} - \lambda_{i_{1}} }
   \cdots \sum_{i_{n}=1}^{N}  \frac{1}{x_{n} - \lambda_{i_{n}} }  \right \rangle_{c} , \quad n \geq 1,
\end{aligned}
\label{resolventb}
\end{gather}
where the notation $\langle \cdot \rangle_{c}$ denotes the fully truncated (connected) average
(for this latter notation see e.g. \cite[Equation~(5.3)]{rmt_Fo}). With this notation, and furthermore
$I = \{ x_{1},\ldots,x_{n-1} \}$, the general loop equation reads (\cite[Equation~(2.19)]{BEMN_2011}
with $a \rightarrow \infty$ and $d/dx_{i}$ corrected to read  $-d/dx_{i}$, \cite[Equation~(2.25)]{BMS_2011} and references therein)
\begin{multline}
  \kappa \sum_{J \subseteq I}W_{|J|+1} (x,J) W_{n- |J|}(x, I \setminus J) + \kappa W_{n+1}(x,x,I)
  + ( \kappa -1) \frac{\partial W_{n}(x,I)} {\partial x  }
\\
  = \frac{\kappa N}{g} \left[ V^{\prime}(x) W_{n}(x,I) - P_{n}(x;I) \right]
  - \sum_{x_{i} \in I} \frac{\partial }{\partial x_{i}}
    \left[ \frac{W_{n-1}(x,I\setminus \{x_{i} \} ) - W_{n-1}(I) } { x - x_{i} } \right] .
\label{LoopEqn}
\end{multline}

Our quantity of interest, $W_{1}(x)$, thus couples with the auxiliary quantities \eqref{resolventb}
for general $n$, and thus in this sense the loop equations are not closed. The utility of the loop
equations reveals itself by hypothesizing that for large $N$, $W_{n}$ has leading term proportional
to $N^{2-n}$, with higher order terms a power series in inverse powers of $N$. Thus one writes (Equation (2.26) \cite{BEMN_2011})
\begin{equation}
\kappa^{n-1} \left( \frac{g}{N} \right)^{2-n} W_{n} =
 \sum_{l=0}^{\infty} \left( \frac{N \sqrt{ \kappa} } {g}  \right)^{-l} W_{n}^{l},
\label{WlargeN}
\end{equation}
where each $W_{n}^{l}$ is independent of $N$, and also (see \S 2.6.2 \cite{BEMN_2011})
\begin{equation}
\kappa^{n-1} \left( \frac{g}{N} \right)^{2-n} P_{n} =
 \sum_{l=0}^{\infty} \left( \frac{N \sqrt{\kappa } }  {g } \right)^{-l} P_{n}^{l},
\label{PlargeN}
\end{equation}
where each $P_{n}^{l}$ is similarly independent of $N$. The partition function has the expansion (see
Equation (2.25) of \cite{BEMN_2011})
\begin{equation}
F = \sum_{ l \ge 0} \left( \frac{\sqrt{\kappa} N} {g} \right)^{2-l} F^{l} .
\label{FlargeN}
\end{equation}
The large $ N $ expansion differs from a genus or semi-classical expansion, wherein the latter would
employ a development in powers of both the expansion parameters
\begin{equation}
\nu = \frac{ N \sqrt{\kappa} } {g}, \quad \hbar = \frac{g}{N} (1- \kappa^{-1}) .
\label{genus_parameters}
\end{equation}

Substituting \eqref{WlargeN} and \eqref{PlargeN} and equating like powers of $N$ one sees that
a quasi-triangular system of equations result, which can be solved recursively. Moreover, one sees too that
the dependence on $\kappa$ is each $W_{n}^{l}$ is a polynomial in $(\sqrt{\kappa} - 1/\sqrt{\kappa})$
of degree $l$,
\begin{multline}
W_{n}^{l} = \left( \sqrt{\kappa} - \frac{1} {\sqrt{\kappa} } \right)^{l} W_{n}^{0,l}
            + \left( \sqrt{\kappa} - \frac{1} {\sqrt{\kappa} }  \right)^{l-2} W_{n}^{1,l-2} +
\\ \cdots
              + \left( \sqrt{\kappa} - \frac{1} {\sqrt{\kappa} }  \right)^{l-2[l/2]} W_{n}^{[l/2],l-2[l/2]},
\label{WdoubleExp}
\end{multline}
where $\{ W_{n}^{g,l - 2g} \}_{g=0,\ldots,[l/2]}$ are independent of $N$ and $\kappa$.

We now specialize to the Gaussian potential $V(x)= \frac{1}{2}x^{2}$ as corresponds to \eqref{GbetaE}.
Since then
\begin{equation}
   \frac{V^{\prime}(x)-V^{\prime}(\lambda)}{x-\lambda} = 1,
\label{GaussianV}
\end{equation}
the definition \eqref{resolventb} gives
\begin{equation}
  P_{1} = \left \langle \sum_{j=1}^{N} 1 \right \rangle = N \langle 1 \rangle =N =: \frac{N}{g} P_{1}^{0} ,
\end{equation}
and thus
\begin{equation}
P_{1}^{l} =
\begin{cases}
g, & l=0, \\
0, & \mbox{otherwise}.
\end{cases}
\label{GaussianP1}
\end{equation}
Also, using \eqref{GaussianV} in the definition \eqref{resolventb} we have
\begin{align*}
P_{n+1}(x;x_{1},\ldots,x_{n}) = \left \langle \sum_{i=1}^{N} 1
\sum_{i_{1}=1}^{N} \frac{1}{x_{1}- \lambda_{i_{1}} } \cdots
\sum_{i_{n}=1}^{N} \frac{1}{x_{n}- \lambda_{i_{n}} }  \right \rangle_{c} .
\end{align*}
In fact this vanishes identically for $ n>1 $.

\begin{theorem}	\label{Thm_GaussianP}
For each $n=1,2,\ldots$ we have
\begin{equation}
P_{n+1}(x;x_{1},\ldots,x_{n}) = 0 .
\end{equation}
\end{theorem}
\begin{proof}
This is a special case of the more general result
\begin{equation}
\langle 1 \cdot A_{1} \cdots A_{n} \rangle_{c} = 0 ,
\label{cumulant_Gaussian}
\end{equation}
valid for $n \ge 1$. We can establish \eqref{cumulant_Gaussian} by induction. The base case is $n=1$
when we have
\begin{align*}
\langle 1 \cdot A_{1} \rangle_{c} := \langle 1 \cdot A_{1} \rangle -
\langle 1 \rangle  \langle  A_{1} \rangle = 0 ,
\end{align*}
as required. We now assume \eqref{cumulant_Gaussian} is valid for $n=1,\ldots,m$, with our remaining task
being to show that it is true for $n=m+1$. For this purpose, let $A_{0}:=1$ and
$A_{I}:=A_{i_{1}} \cdots A_{i_{n}}$, $I = \{ i_{1}, \cdots, i_{k} \}$. Then from the definition of a connected
correlator we have
\begin{equation}
\langle 1 \cdot A_{1} \cdots A_{m+1} \rangle
= \langle 1 \cdot A_{1} \cdots A_{m+1} \rangle_{c}
+ \sum_{k=2}^{m+2} \sum_{I_{1} \cup \cdots \cup I_{k} = \{0,\ldots,m+1 \} }
\prod_{j=1}^{k}  \langle A_{I_{j}} \rangle_{c} .
\label{AUXcumulant}
\end{equation}
By the induction hypothesis, if $0 \in I_{j}$ and $|I_{j}|>1$ we have $ \langle A_{I_{j}} \rangle_{c} =0 $
(i.e. the 0 index must be in a subset of its own). Thus
\begin{align*}
 \sum_{k=2}^{m+2} \sum_{I_{1} \cup \cdots \cup I_{k} = \{0,\ldots,m+1 \} }
\prod_{j=1}^{k}  \langle A_{I_{j}} \rangle_{c}
& =
 \sum_{k=1}^{m+1} \sum_{I_{1} \cup \cdots \cup I_{k} = \{1,\ldots,m+1 \} }
\prod_{j=1}^{k}  \langle A_{I_{j}} \rangle_{c} \notag \\
& = \langle A_{1} \cdots A_{m+1} \rangle.
\end{align*}
Substituting this back in \eqref{AUXcumulant} implies \eqref{cumulant_Gaussian}.
\end{proof}

The loop equations, for the Gaussian potential and resolved into the large $N$ expansion, are then
solved in the following hierarchy, whose initial parts we give in three steps. -\\
{\it Step 1} -
Order $N^2$ terms of the $n=1$ loop equations:\\
\begin{equation*}
(W_{1}^{0}(x))^{2} - x W_{1}^{0}(x) + g = 0 ,
\end{equation*}
with the solution
\begin{equation*}
W_{1}^{0}(x) = \tfrac{1}{2} \left( x - \sqrt{x^{2} - 4g} \right) ,
\label{W1,0_soln}
\end{equation*}
where the negative sign is chosen so that  $W_{1}^{0} \mathop{\sim}_{x \rightarrow \infty} O(x^{-1})$. Thus
\begin{equation*}
W_{1}^{0}  =  \tfrac{1}{4}  \sum_{k \ge 0}  \frac{(4g)^{k+1}}{x^{2k+1}}\frac{\Gamma(k +  \frac{1}{2}) } {(k+1)! \Gamma( \frac{1}{2}) } ,
\label{W1,0_expansion}
\end{equation*}
which yield the moments of the Wigner semi-circle law
\begin{equation*}
\tilde{\rho}_{(1),0}(x;\beta) = \frac{1}{2\pi} \sqrt{4g - x^{2}} \quad \mbox{on } (-2\sqrt{g}, 2\sqrt{g}) .
\label{rho0}
\end{equation*}
{\it Step 2} -
Order $N^{1}$ terms of the $ n=1 $ loop equations:\\
\begin{equation*}
 (2W_{1}^{0}(x)-x) W_{1}^{1}(x)  + \left( \sqrt{\kappa}-\frac{1}{\sqrt{\kappa}} \right) \frac{\partial}{\partial x}W_{1}^{0}(x) = 0 ,
\end{equation*}
with the solution
\begin{equation*}
W_{1}^{1}(x) = \tfrac{1}{2} \left( \sqrt{\kappa}-\frac{1}{\sqrt{\kappa}}  \right)
	\left[ \frac{1}{\sqrt{x^{2} -4g } } - \frac{x}{x^{2} -4g} \right] .
\label{W1,1_soln}
\end{equation*}
Its large $x$ expansion is
\begin{equation*}
W_{1}^{1} = \tfrac{1} {2}  \left( \sqrt{\kappa}-\frac{1}{\sqrt{\kappa}}  \right)
	\sum_{k \ge 0} \frac{(4g)^{k} } {x^{2k+1} } \left[ \frac{(\frac{1}{2})_{k}}{k!}  -1  \right]   \notag \\
\sim  - \left( \sqrt{\kappa}-\frac{1}{\sqrt{\kappa}}  \right) g x^{-3} + \cdots .
\label{W1,1_expansion}
\end{equation*}
{\it Step 3} -
Order $N^{1}$ terms of $n=2$ loop equation and order $N^{0}$ terms of the $ n=1 $ loop equations:\\
\begin{equation*}
  \left(2W_{1}^{0}(x)-x \right) W_{2}^{0}(x,x_{1}) + \frac{\partial}{\partial x_{1}} \frac{W_{1}^{0}(x)-W_{1}^{0}(x_{1})}{x-x_{1}} = 0 ,
\end{equation*}
with its solution
\begin{equation*}
W_{2}^{0}(x,x_{1})
   =  \tfrac{1}{2} (x-x_{1})^{-2} \left[ \frac{\sqrt{x_{1}^{2}-4g}}{\sqrt{x^{2}-4g}}-1 \right]
     +\tfrac{1}{2} (x-x_{1})^{-1} \frac{x_{1}}{\sqrt{x_{1}^{2}-4g}\sqrt{x^{2}-4g}} ,
\label{W2,0_soln}
\end{equation*}
followed by
\begin{equation*}
 (2W_{1}^{0}(x)-x)W_{1}^{2}(x) + W_{2}^{0}(x,x) + (W_{1}^{1}(x))^{2} + \left( \sqrt{\kappa}-\frac{1}{\sqrt{\kappa}} \right)\frac{\partial}{\partial x}W_{1}^{1}(x) = 0 ,
\end{equation*}
and its solution
\begin{equation*}
W_{1}^{2}(x) =  \left( \sqrt{\kappa}-\frac{1}{\sqrt{\kappa}}  \right)^{2}
	\left[ -\frac{x}{(x^{2}-4g)^{2}} + \frac{x^{2} + g}{(x^{2}-4g)^{5/2}} \right] + \frac{g}{(x^{2}-4g)^{5/2}} .
\label{W1,2_soln}
\end{equation*}
This has the large $x$ expansion
\begin{equation*}
W_{1}^{2}(x) \mathop{\sim}_{x \rightarrow \infty} \left( \sqrt{\kappa}-\frac{1}{\sqrt{\kappa}} \right)^{2}  \left( \frac{3g}{x^{5}} + \cdots \right)
	+ \frac{g}{x^{5}} + \cdots .
\label{W1,2_expansion}
\end{equation*}
In \S \ref{DataLargeN} we supplement our computation of $W_1^0$, $W_1^1$, $W_1^2$ by
specifying $W_1^{(j)}$ for $j$ up to 6, and we furthermore use this to compute the asymptotic
smoothed densities and the moments.

\section{Expansions of the resolvent, moments and smoothed densities for general $ \beta $}\label{DataLargeN}
\setcounter{equation}{0}

\subsection{Resolvent expansion}
We record here the results of the large $ N $ expansion of the resolvent as computed using the loop
equations given in \S \ref{LoopEquations} and make a number of observations on this data.
We recall from (\ref{WlargeN})
that this expansion has the form
\begin{equation}
\frac{g}{N}\tilde{W}_{1} = W_{1}^{0} + \left(\frac{g} {\sqrt{\kappa}N}\right)W_{1}^{1} + \left(\frac{g} {\sqrt{\kappa}N}\right)^{2} W_{1}^{2}
           + \left(\frac{g}{\sqrt{\kappa}N}\right)^{3} W_{1}^{3} +\ldots \quad,
\label{resolventN}
\end{equation}
which as noted in the Introduction is known to be rigorously valid for the
G$\beta$E${}^*(g)$ ensembles.
Relative to (\ref{WlargeN}), on the LHS we have written $\tilde{W}_1$ in place of $W_1$,
so we can distinguish the expanded form from the definition (\ref{resolvent}).
We will utilise the following abbreviations for the variable characterising the $ \beta $-deformation
or deviation from the hermitian case, and the single-cut spectral curve
\begin{equation}\label{hy}
   h:= \sqrt{\kappa} - \frac {1} {\sqrt{\kappa}}, \quad
   y(x):= \sqrt{x^{2}-4g} .
\end{equation}
In the compact two-term form the first six coefficients are
\begin{equation}
W_{1}^{0} = \frac{1}{2} \left[ x - y \right] ,
\label{resolvent_0}
\end{equation}
\begin{equation}
W_{1}^{1} = h \frac{1}{2} \left[ \frac{1}{y} - \frac{x}{y^2} \right] ,
\label{resolvent_1}
\end{equation}
\begin{equation}
W_{1}^{2} = h^{2} \left[ -\frac{x}{y^{4}} + \frac{x^{2} + g } {y^{5}} \right] + \frac{g } {y^{5}} ,
\label{resolvent_2}
\end{equation}
\begin{equation}
W_{1}^{3} = h^{3} 5 \left[ \frac{x^{2} + g } {y^{7}} - \frac{x^{3} + 2gx} {y^{8}} \right]
            +h \frac{1}{2} \left[  \frac{x^{2} + 6g} {y^{7} } - \frac {x^{3} + 30gx} { y^{8}} \right] ,
\label{resolvent_3}
\end{equation}
\begin{multline}
W_{1}^{4} = h^4\left[ -\frac{37 x^3+92 g x}{y^{10}} + \frac{37 x^4+123 g x^2+21 g^2}{y^{11}} \right]
\\
            +h^2\left[ -\frac{23 x^3+180 g x}{2 y^{10}} + \frac{23 x^4+454 g x^2+176 g^2}{2 y^{11}} \right]
            +\frac{21 g \left( x^2+g\right)}{y^{11}} ,
\label{resolvent_4}
\end{multline}
\begin{multline}
W_{1}^{5} =h^5\left[ \frac{353 x^4+1527 g x^2+399 g^2}{y^{13}}-\frac{353 x^5+1766 g x^3+848 g^2 x}{y^{14}} \right]
\\
            +h^3\left[ \frac{445 x^4+4332 g x^2+1512 g^2}{2y^{13}}-\frac{445 x^5+7714 g x^3+7440 g^2 x}{2y^{14}} \right]
\\
            +h\left[ \frac{21 \left(x^4+20 g x^2+14 g^2\right)}{2y^{13}}-\frac{3 \left(7 x^5+628 g x^3+1200 g^2 x\right)}{2y^{14}} \right] ,
\label{resolvent_5}
\end{multline}
\begin{multline}
 W_{1}^{6} = h^6\left[ -\frac{4081 x^5+26392 g x^3+18976 g^2 x}{y^{16}}+\frac{4081 x^6+28625 g x^4+26832 g^2 x^2+1738g^3 }{y^{17}} \right]
\\
            +h^4\left[ -\frac{8567 x^5+101288 g x^3+93600 g^2x}{2y^{16}}+\frac{8567 x^6+147556 g x^4+243180 g^2 x^2+31236 g^3}{2y^{17}} \right]
\\
            +h^2\left[ -\frac{618 x^5+13104 g x^3+18000 g^2 x}{y^{16}}+\frac{618 x^6+32043 g x^4+91299 g^2 x^2+16834 g^3}{y^{17}} \right]
\\
            +\frac{11 g (135 x^4+558 g x^2+158 g^2)}{y^{17}} .
\label{resolvent_6}
\end{multline}


\begin{remark}
Essentially the same number of coefficients were reported in Eq. (33) of \cite{MMPS_2012}, which agree with
our results after correcting for the typographical errors in $ \rho_{1,5} $. Partial results have also been
given in Eq. (2.60) of \cite{BMS_2011} up to $ W_{1}^{4} $ but their result for $ W_{1,2}(p) $ differs from
the coefficient of $ h^2 $ in \eqref{resolvent_4}. Our results, specialised to $ \kappa = 1 $, are consistent
with the recurrence system \eqref{GUEeta}, \eqref{GUErecurC} given in \cite{HT_2012} and the expansion, Eq. (3.60),
of \cite{MS_2009}.
\end{remark}

Inspection of (\ref{resolvent_0}) to (\ref{resolvent_6}) suggest the following analytic form for the $W_1^l$.
\begin{conjecture}\label{C1}
Let $y = y(x)$ be given by (\ref{hy}).
For $ l\geq 2 $ even we have
\begin{equation}
   W_{1}^{l}(x) = h^l \left[ \frac{P^{l}_{1}(x)}{y^{3l-2}}+\frac{P^{l}_{2}(x)}{y^{3l-1}} \right]
                + h^{l-2} \left[ \frac{P^{l}_{3}(x)}{y^{3l-2}}+\frac{P^{l}_{4}(x)}{y^{3l-1}} \right] +
\\
          \ldots + h^2 \left[ \frac{P^{l}_{l-1}(x)}{y^{3l-2}}+\frac{P^{l}_{l}(x)}{y^{3l-1}} \right]
                   + \frac{P^{l}_{l+1}(x)}{y^{3l-1}} ,
\label{Wl_even}
\end{equation}
where $ {\rm deg}_{x}P^{l}_{j}=l-1 $ for $ j=1,3,\ldots, l-1 $, $ {\rm deg}_{x}P^{l}_{j}=l $ for $ j=2,4,\ldots, l $
and $ {\rm deg}_{x}P^{l}_{l+1}=l-2 $.
For $ l\geq 1 $ odd we have
\begin{equation}
   W_{1}^{l}(x) = h^l \left[ \frac{P^{l}_{1}(x)}{y^{3l-2}}+\frac{P^{l}_{2}(x)}{y^{3l-1}} \right]
                + h^{l-2} \left[ \frac{P^{l}_{3}(x)}{y^{3l-2}}+\frac{P^{l}_{4}(x)}{y^{3l-1}} \right] +
          \ldots + h \left[ \frac{P^{l}_{l}(x)}{y^{3l-2}}+\frac{P^{l}_{l+1}(x)}{y^{3l-1}} \right] ,
\label{Wl_odd}
\end{equation}
where the polynomial numerators have $ {\rm deg}_{x}P^{l}_{j}=l-1 $ for $ j=1,3,\ldots, l $ and $ {\rm deg}_{x}P^{l}_{j}=l $
for $ j=2,4,\ldots, l+1 $. The polynomial numerators $ P^{l}_{j} $ are either even or odd with respect to $ x\mapsto -x $
according as the degree is even or odd respectively.
Furthermore, the leading term in the $ x \to \infty $ expansion of $ W_{1}^{l}(x) $ is of order $ x^{-2l-1} $ for all $ l\geq 0 $.
\end{conjecture}

\subsection{Expansion of the smoothed density}
Having the resolvent at hand, in the form of a development in descending powers of $ N $, we come the extract meaning to the density via the inversion
Sokhotski-Plemelj formula
\begin{equation}
  \tilde{\rho}_{(1)}(x) = \frac{1}{2\pi i}\left[ \tilde{W}_1(x-i\epsilon)-\tilde{W}_1(x+i\epsilon) \right]_{x\in (-2\sqrt{g},2\sqrt{g})} .
\end{equation}
However by using the large $ N $ expansion for $ W_1(x) $ this formulae does not yield the true density,
as we have indicated by our notation, but rather the {\it smoothed density} $ \tilde{\rho}_{(1)}(x) $.
The smoothed density does not possess any of the oscillatory contributions of the true density,
the leading order contributions of which have been found in a number of studies
(see e.g.~\cite{GFF_2005, FFG_2006, DF_2006a}), but rather the remnant of these when integrated against classes of test
functions (usually continuously differentiable of all orders and bounded functions $ C^{\infty}_b $). The
smoothed density is in fact a distribution with respect to such a class of functions.

To facilitate the extraction of the smoothed density $ \tilde{\rho}_{(1)}(x) $ it is necessary to express the
coefficients of \eqref{resolventN} in a
partial-fraction form with terms containing factors of $ y^{-\sigma} $, $ x y^{-\sigma} $ where
$ \sigma \in \mathbb{Z}, \mathbb{Z}+\frac{1}{2} $. Noting the above formula
and \eqref{resolventN} we have a similar large-$N$ expansion (although defined slightly differently from \eqref{largeNmu}, in which case there is no variable $g$)
\begin{equation}
 \frac{g}{N}\tilde{\rho}_{(1)} = \tilde{\rho}_{(1),0} + \left(\frac{g} {\sqrt{\kappa}N}\right)\tilde{\rho}_{(1),1} + \left(\frac{g} {\sqrt{\kappa}N}\right)^{2} \tilde{\rho}_{(1),2}
           + \left(\frac{g}{\sqrt{\kappa}N}\right)^{3} \tilde{\rho}_{(1),3} +\ldots \quad.
\label{densityN}
\end{equation}
In addition to the indicator or step-function $ \chi_{x \in (-2\sqrt{g},2\sqrt{g})} $ let us define the Dirac
delta distributions
\begin{equation}
   \epsilon^{(l)}_{(-2\sqrt{g},2\sqrt{g})}(x) := \delta^{(l)}(x-2\sqrt{g})+(-1)^l\delta^{(l)}(x+2\sqrt{g}) ,\quad l = 1,2, \ldots ,
\end{equation}
where we note that the first of these
\begin{equation}
   \epsilon^{(1)}_{(-2\sqrt{g},2\sqrt{g})} = \frac{d}{dx}\chi_{x \in (-2\sqrt{g},2\sqrt{g})} .
\end{equation}
Using a partial fraction expansion of the coefficients in \eqref{densityN} along with
\eqref{resolvent_0}-\eqref{resolvent_6} we have
\begin{equation}
  \tilde{\rho}_{(1),0}(x) = \frac{1}{2\pi}\sqrt{4g-x^2}\;\chi_{x \in (-2\sqrt{g},2\sqrt{g})},
\label{density0}
\end{equation}
\begin{equation}
  \tilde{\rho}_{(1),1}(x) = h \left\{ \frac{1}{2\pi}(4g-x^2)^{-1/2}\;\chi_{x \in (-2\sqrt{g},2\sqrt{g})}
                     - \frac{1}{4}\epsilon^{(0)}_{(-2\sqrt{g},2\sqrt{g})} \right\} ,
\label{density1}
\end{equation}
\begin{multline}
  \tilde{\rho}_{(1),2}(x) =
  h^2 \left\{ \frac{1}{\pi}(x^2+g)(4g-x^2)^{-5/2}\;\chi_{x \in (-2\sqrt{g},2\sqrt{g})}
           +\frac{1}{8\sqrt{g}}\epsilon^{(1)}_{(-2\sqrt{g},2\sqrt{g})} \right\}
\\
                     +\frac{1}{\pi}g(4g-x^2)^{-5/2}\;\chi_{x \in (-2\sqrt{g},2\sqrt{g})} ,
\label{density2}
\end{multline}
\begin{multline}
  \tilde{\rho}_{(1),3}(x) =
   h^3 \left\{ -\frac{5}{\pi}(x^2+g)(4g-x^2)^{-7/2}\;\chi_{x \in (-2\sqrt{g},2\sqrt{g})}
       \right.
\\     \left.
                -\frac{5}{512g^{3/2}}\epsilon^{(1)}_{(-2\sqrt{g},2\sqrt{g})}
                -\frac{5}{256g}\epsilon^{(2)}_{(-2\sqrt{g},2\sqrt{g})}
                -\frac{5}{128g^{1/2}}\epsilon^{(3)}_{(-2\sqrt{g},2\sqrt{g})}
       \right\}
\\
   + h \left\{ -\frac{1}{2\pi}(x^2+6g)(4g-x^2)^{-7/2}\;\chi_{x \in (-2\sqrt{g},2\sqrt{g})}
       \right.
\\     \left.
                +\frac{13}{1024g^{3/2}}\epsilon^{(1)}_{(-2\sqrt{g},2\sqrt{g})}
                +\frac{13}{512g}\epsilon^{(2)}_{(-2\sqrt{g},2\sqrt{g})}
                 +\frac{17}{768g^{1/2}}\epsilon^{(3)}_{(-2\sqrt{g},2\sqrt{g})}
       \right\} .
\label{density3}
\end{multline}
We present $  \tilde{\rho}_{(1),4}(x)$, $  \tilde{\rho}_{(1),5}(x) $ and $  \tilde{\rho}_{(1),6}(x) $
in the Appendix.

Some comments are in order regarding the meaning of these results in relation to their use in
\eqref{largeNmu}. The correct meaning of the integral in \eqref{largeNmu} is the Hadamard
regularised form, or the {\it partie finie} \cite{Hadamard_1953}, which was shown by Riesz \cite{Riesz_1938}
to be the meromorphic continuation of a finite integral. Here we indicate this with the
relevant example for the power law singularities of $\tilde{\rho}_{(1),g}$ at $x = \pm 1$,
\begin{equation}\label{Ia}
   I(\alpha) := \int^{1}_{-1} dx f(x) (1-x^2)^{-\alpha} ,
\end{equation}
where $ f(x) $ is continuously differentiable up to order $ p+1 $ for $ x\in (0,1) $, and $ \alpha \in \mathbb{R}> 0 $.
Changing variables $ y=x^2 $ and defining $ F(y)=\frac{1}{2}[ f(\sqrt{y})+f(-\sqrt{y}) ] $ we subtract off the
first $ p+1 $ terms of the Taylor expansion of $ F $ in the integrand giving
\begin{multline}\label{ML}
  I(\alpha) = \int^{1}_{0} dy\; y^{-1/2}(1-y)^{-\alpha}
              \left[ F(y)-F(1)-(y-1)F'(1)- \cdots -\frac{1}{p!}(y-1)^pF^{(p)}(1) \right]
\\
            + F(1)\int^{1}_{0}dy\;y^{-1/2}(1-y)^{-\alpha} + \ldots
             + \frac{(-1)^p}{p!}F^{(p)}(1)\int^{1}_{0}dy\;y^{-1/2}(1-y)^{p-\alpha} ,
\end{multline}
where the first integral is clearly an ordinary integral if $ p-\alpha> -2 $ and the latter integrals are to be Hadamard regularised.
These latter integrals are examples of Euler $\beta$ integrals and can be evaluated according to
\begin{equation}\label{EB}
   \int^{1}_{0}dy\;y^{-1/2}(1-y)^{q-\alpha} = \frac{\Gamma(1/2)\Gamma(q-\alpha+1)}{\Gamma(q-\alpha+3/2)} ,
\end{equation}
where the meromorphic continuation is with respect to $ \alpha $ and through the explicit form of
the Gamma function.

It is  furthermore the case that  for the singularities of
$\tilde{\rho}_{(1),g}$,
$ \alpha $ in (\ref{Ia}) is a positive half-integer $ \alpha=n+\tfrac{1}{2} $. Then the denominator of the
right-hand side of (\ref{EB})
is $ \Gamma(p-n+1) $.  Also, taking
$p=n-1$ leaves us with a convergent integral in the first line of (\ref{ML}). But for
$ 0\leq p\leq n-1 $ the argument of $ \Gamma(p-n+1) $ is a negative integer
and thus the finite-part is actually zero (the numerator Gamma functions have half-integer arguments).
Thus in the sense of Hadamard regularisation we have ($ p=n-1 $)
\begin{equation}\label{EB1}
  I(n + \tfrac{1}{2}) = \int^{1}_{0} dy\; y^{-1/2}(1-y)^{-\alpha}
              \left[ F(y)-F(1)-(y-1)F'(1)- \cdots -\frac{1}{p!}(y-1)^pF^{(p)}(1) \right] .
\end{equation}

An alternative understanding of (\ref{EB1}) is possible. First we note that the final statement
in Conjecture \ref{C1} is equivalent to the moment identity
\begin{equation}
   \int^{\infty}_{-\infty} dx\; x^{2\sigma} \tilde{\rho}_{(1),l}(x) = 0, \quad 0 \leq \sigma \leq l-1, \quad l \geq 1 .
\label{ORTHOGdensity}
\end{equation}
It can be shown from the explicit forms of \eqref{density1}-\eqref{density3} and \eqref{density4}-\eqref{density6}
that \eqref{ORTHOGdensity}
is satisfied for $ 0 \leq \sigma \leq l-1 $ and $ 1 \leq l \leq 6 $. The mechanism of how this occurs is
that there is mutual cancellation amongst the terms with delta function derivatives for low moment orders
$ 0 \leq \sigma \leq l-2 $, whilst the cancellation at $ \sigma=l-1 $ is between the first non-zero integral
and the delta-function terms. Use of (\ref{ORTHOGdensity}) shows that subtracting the first
$l-1$ terms  of the power terms in the variable $1 - x^2$ of an averaged function $f(x)$ leaves
the average unchanged and moreover transforms the divergent integral to a convergent one.
%

\subsection{Moments}
For purposes of comparison with earlier works in this section we will specialise to the
ensemble G$\beta$E$(N)$ (recall text below (\ref{pdfSTAR})) with the PDF
\begin{equation}
    \prod_{l=1}^{N} e^{-\frac{1}{2}\kappa \lambda_{l}^2} \prod_{1 \le j < k \le N}|\lambda_{k}-\lambda_{j} |^{2\kappa}  .
\label{GbetaEpdf}
\end{equation}
Thus we have $ m_{2p}(N,\kappa) := \left\langle \Tr G^{2p} \right\rangle_{{\rm G{\beta}E}} $ in comparison to
our earlier definition \eqref{resolvent}
\begin{equation}
    m^{*}_{2l}(N,\kappa) = \left(\frac{g}{N}\right)^{l}m_{2l}(N,\kappa) ,\quad l\in \mathbb{Z}_{\geq 0}.
\label{Mscaling}
\end{equation}

Moments of the density can be readily computed from the resolvent coefficients \eqref{resolvent_1}-\eqref{resolvent_6}.
Thus we find
\begin{align}
m_{0} = & N ,
\label{moment0} \\
m_{2} = & N^2+N \left(-1+\kappa^{-1}\right) ,
\label{moment2} \\
m_{4} = & 2 N^3+5N^2 \left(-1+\kappa^{-1}\right)+N \left(3-5\kappa^{-1}+3\kappa^{-2}\right) ,
\label{moment4} \\
m_{6} = & 5 N^4+22 N^3 \left(-1+\kappa^{-1}\right)+N^2 \left(32-54\kappa^{-1}+32\kappa^{-2}\right)
\label{moment6} \\
      &  +N\left(-15+32\kappa^{-1}-32\kappa^{-2}+15\kappa^{-3}\right) ,
\notag\\
m_{8} = & 14 N^5+93 N^4 \left(-1+\kappa^{-1}\right)+N^3 \left(234-398\kappa^{-1}+234\kappa^{-2}\right)
\label{moment8} \\
      &  +N^2\left(-260+565\kappa^{-1}-565\kappa^{-2}+260\kappa^{-3}\right)
\notag\\
      &  +N\left(105-260\kappa^{-1}+331\kappa^{-2}-260\kappa^{-3}+105\kappa^{-4}\right) ,
\notag\\
m_{10} = & 42 N^6+386 N^5 \left(-1+\kappa^{-1}\right)+10 N^4\left(145-248\kappa^{-1}+145\kappa^{-2}\right)
\label{moment10} \\
      &  +550 N^3\left(-5+11\kappa^{-1}-11\kappa^{-2}+5\kappa^{-3}\right)
\notag\\
      &  +N^2\left(2589-6545\kappa^{-1}+8395\kappa^{-2}-6545\kappa^{-3}+2589\kappa^{-4}\right)
\notag\\
      &  +N\left(-945+2589\kappa^{-1}-3795\kappa^{-2}+3795\kappa^{-3}-2589\kappa^{-4}+945\kappa^{-5}\right) ,
\notag\\
m_{12} = & 132 N^7+1586 N^6\left(-1+\kappa^{-1}\right)+N^5 \left(8178-14046\kappa^{-1}+8178\kappa^{-2}\right)
\label{moment12} \\
         &  +N^4\left(-22950+50945\kappa^{-1}-50945\kappa^{-2}+22950\kappa^{-3}\right)
\notag\\
         &  +4 N^3\left(9125-23403\kappa^{-1}+30173\kappa^{-2}-23403\kappa^{-3}+9125\kappa^{-4}\right)
\notag\\
         &  +N^2\left(-30669+85796\kappa^{-1}-127221\kappa^{-2}+127221\kappa^{-3}-85796\kappa^{-4}+30669\kappa^{-5}\right)
\notag\\
         &  +3 N\left(3465-10223\kappa^{-1}+16432\kappa^{-2}-18853\kappa^{-3}+16432\kappa^{-4}-10223\kappa^{-5}+3465\kappa^{-6}\right) .
\notag
\end{align}

\begin{remark}
Results for moments up to $ m_{6} $ were given in \cite{DE_2006}, see pg. 9 of that work, and up to $ m_{8} $
were also given by Eq.(24) in \cite{MMPS_2012}, both sets of which coincide with our calculations. We have
also used the MOPS package \cite{DES_2007} to compute the moments up to $ m_{20} $ and find that the first
six coincide with those given above. Another four moments are recorded in \eqref{moment14}-\eqref{moment20}.
\end{remark}

In the study of Dimitriu and Edelman \cite{DE_2006} structural properties for the moments were established using
Jack polynomial theory, and in particular we have the following result.
\begin{theorem}[Thm 2.8 of \cite{DE_2006}]
The general moment $ m_{2l}(N,\kappa) $, $ l\geq 0 $ is a polynomial of degree $ l+1 $ in $ N $ and
has a vanishing tail coefficient, i.e. is proportional to $ N $.
The coefficients with respect to $ N $ are polynomials in $ \kappa^{-1} $ with degree increasing linearly by
unity from the leading term whose degree is zero. These coefficients have a numerator which are a palindromic
polynomial in $ \kappa $ if of even degree or an anti-palindromic polynomial if of odd degree. In the latter case
the numerator has a factor of $ \kappa-1 $. This property can be expressed by the duality relation which
\begin{equation}
  m_{2l}(N,\kappa) = (-1)^{l+1}\kappa^{-l-1}m_{2l}(-\kappa N, \kappa^{-1}) , \quad \forall\; l \geq 0, \;\kappa > 0 .
\label{full_duality}
\end{equation}
\end{theorem}
\begin{remark}
It is immediate that
\eqref{moment0}-\eqref{moment12} satisfy \eqref{full_duality}.
\end{remark}

In addition to the resolvent $W_1(x) = W_1(x,N,\kappa)$, let us introduce the exponential
generating function
\begin{equation}
 u(t,N,\kappa) \equiv \sum_{p=0}^{\infty} \frac{t^{2p}}{(2p)!} \left\langle \Tr G^{2p} \right\rangle_{{\rm G{\beta}E}}
       = \left \langle \sum_{j=1}^{N} e^{t \lambda_{j} } \right \rangle_{{\rm G{\beta}E}}
       = \left \langle \Tr e^{tG} \right \rangle_{{\rm G{\beta}E}} .
\label{EXPGF}
\end{equation}
The formal relation with the resolvent, as defined by \eqref{pdfSTAR} and \eqref{resolvent}, is
\begin{gather}
\int_{0}^{\infty} dt\; e^{-xt} u(t,N) =\sqrt{\frac{g}{N}}W_{1}(\sqrt{\frac{g}{N}}x,N) ,
\label{laplaceXfm}
\end{gather}
but in line with our earlier remarks the existence of the integral needs to be examined.
Other generating functions have been employed, including a ``sub''-exponential type defined by
\begin{equation}
   \phi(s,N) \equiv \sum_{p=0}^{\infty} \frac{s^{2p}}{(2p-1)!!} \left \langle \Tr G^{2p} \right \rangle_{{\rm G{\beta}E}} .
\label{subEXPGF}
\end{equation}
It will be observed that \eqref{subEXPGF} is a convergent sum, and both \eqref{EXPGF} and \eqref{subEXPGF}
can be seen as Borel resummations of the divergent expansion for the resolvent.

From (\ref{full_duality}) we can immediately deduce the consequences for the generating functions themselves.
\begin{corollary}\label{dual}
The generating functions satisfy the following duality relations for $ \kappa, N > 0 $
\begin{gather}
    W_{1}(x,N,\kappa) = -\kappa^{-1}W_{1}(x,-\kappa N,\kappa^{-1}) ,
\label{Dual_W}\\
    u(t,N,\kappa) = -\kappa^{-1}u(\kappa^{-1/2}it,-\kappa N,\kappa^{-1}) ,
\label{Dual_u}\\
    \phi(s,N,\kappa) = -\kappa^{-1}\phi(\kappa^{-1/2}is,-\kappa N,\kappa^{-1}) .
\label{Dual_phi}
\end{gather}
\end{corollary}
\begin{remark}
One can verify that \eqref{resolventN} together with \eqref{Wl_even} and \eqref{Wl_odd} satisfies \eqref{Dual_W}
for general $ x, N, \kappa $.
\end{remark}

To the foregoing result on the low order moments we can add some explicit detail concerning the coefficients
of the general $ 2l$-th moment. It is a classical result that
the leading coefficient with respect to $ N $ is the $ l $-th Catalan number,
$ C_{l} \equiv \frac{(2l)!}{l!(l+1)!} $, which expresses the appearance
of the Wigner semi-circle law in the bulk scaling of $ \frac{1}{N}W_{1}(x,N) $ in the $ N \to \infty $ limit.
We can give simple explicit formulae for the sub-leading coefficients as well, for general values of $ l $.
\begin{theorem}\label{T19}
The moment $ m_{2l} $ has the further properties - \\
The coefficient of $ \kappa^{-1/2}N^{l} $ in $ m_{2l} $ is (which is the integer sequence {\tt A000346} \cite{OEIS_2010})
\begin{equation}
   2^{2l-1}\left[ -1+\frac{\Gamma(l+\frac{1}{2})}{\sqrt{\pi}\Gamma(l+1)} \right]h, \quad l\geq 1 .
\label{Mcoeff:1st}
\end{equation}
The coefficient of $ \kappa^{-1}N^{l-1} $ in $ m_{2l} $ is
\begin{equation}
  \frac{1}{3}4^{l-1} l\left[ -3+(5l+1)\frac{\Gamma(l+\frac{1}{2})}{\sqrt{\pi}\Gamma(l+1)} \right]h^2
     + \frac{1}{3}4^{l-1}l(l-1) \frac{\Gamma(l+\frac{1}{2})}{\sqrt{\pi}\Gamma(l+1)}, \quad l\geq 2 .
\label{Mcoeff:2nd}
\end{equation}
The coefficient of $ \kappa^{-3/2}N^{l-2} $ in $ m_{2l} $ is
\begin{equation}
  \frac{5}{3}4^{l-3} l^2(l-1) \left[ -3+\frac{8 \Gamma(l+\frac{1}{2})}{\sqrt{\pi}\Gamma(l+1)} \right]h^3
    + \frac{1}{3}2^{2l-7}l(l-1)\left[ 28-17l+\frac{16(l-1) \Gamma(l+\frac{1}{2})}{\sqrt{\pi}\Gamma(l+1)} \right]h , \quad l\geq 3 .
\label{Mcoeff:3rd}
\end{equation}
The coefficient of $ \kappa^{-2}N^{l-3} $ in $ m_{2l} $ is
\begin{multline}
  2^{2l-7} l(l-1)(l-2) \left[ \frac{1}{3}(8-15l)+\frac{4(1105l^2-193l-42) \Gamma(l+\frac{1}{2})}{945\sqrt{\pi}\Gamma(l+1)}\right]h^4
\\
    +  4^{l-4} l(l-1)(l-2)\left[ \frac{1}{3}(28-17l) +\frac{16(590l^2-1259l-84) \Gamma(l+\frac{1}{2})}{945\sqrt{\pi} \Gamma(l+1)} \right] h^2
\\
      + 2^{2l-5}l(l-1)(l-2)(l-3) \frac{(5l-2) \Gamma(l+\frac{1}{2})}{45\sqrt{\pi} \Gamma(l+1)}, \quad l\geq 4 .
\label{Mcoeff:4th}
\end{multline}
The coefficient of $ \kappa^{-5/2}N^{l-4} $ in $ m_{2l} $ is
\begin{multline}
  2^{2l-13} l^2(l-1)(l-2)(l-3) \left[ \frac{1}{3}(99-113l)+\frac{128(1105l-1243)\Gamma(l+\frac{1}{2})}{945\sqrt{\pi}\Gamma(l+1)} \right] h^5
\\
    +  4^{l-7} l(l-1)(l-2)(l-3)
\\ \times
    \left[ -\frac{1}{45}(5677l^2-17271l+4952)+\frac{(302080l^2-698368l+10752) \Gamma(l+\frac{1}{2})}{945\sqrt{\pi}\Gamma(l+1)} \right] h^3
\\
      + 2^{2l-13} l(l-1)(l-2)(l-3)
\\ \times
      \left[ -\frac{1}{15}(l-1)(239l-886)+\frac{128(l-3)(5l-2) \Gamma(l+\frac{1}{2})}{45\sqrt{\pi}\Gamma(l+1)} \right] h, \quad l\geq 5 .
\label{Mcoeff:5th}
\end{multline}
The coefficient of $ \kappa^{-3}N^{l-5} $ in $ m_{2l} $ is
\begin{multline}
   4^{l-7} l(l-1)(l-2)(l-3)(l-4)
\\ \times
   \left[ -\frac{1}{15}(565l^2-1295l+512)+\frac{128(82825 l^3-135690 l^2+8081 l+1716) \Gamma(l+\frac{1}{2})}{405405\sqrt{\pi} \Gamma(l+1)} \right] h^6
\\
    +  2^{2 l-15} l(l-1)(l-2)(l-3)(l-4)
\\ \times
       \left[ -\frac{1}{45}(5677l^2-19991l+9432)+\frac{256(5929 l^3-23320 l^2+12861 l+312) \Gamma(l+\frac{1}{2})}{12285\sqrt{\pi} \Gamma(l+1)} \right] h^4
\\
      +  4^{l-7} l(l-1)(l-2)(l-3)(l-4)
\\ \times
         \left[ -\frac{1}{15}(l-1)(239l-886)+\frac{128(93427 l^3-549765 l^2+623360 l+9438) \Gamma(l+\frac{1}{2})}{405405\sqrt{\pi} \Gamma(l+1)} \right] h^2
\\
        + 2^{2l-7} l(l-1)(l-2)(l-3)(l-4)(l-5) \frac{(35 l^2-77 l+12) \Gamma(l+\frac{1}{2})}{2835\sqrt{\pi} \Gamma(l+1)} , \quad l\geq 6 .
\label{Mcoeff:6th}
\end{multline}
\end{theorem}
\begin{proof}
The first two equalities follow from Theorems \ref{GOEfirst} and \ref{GSEfirst}, and the previous theorem -
see the formulae \eqref{GOE_expN} and \eqref{GSE_expN}.
All of the relations \eqref{Mcoeff:1st}-\eqref{Mcoeff:6th} can be established by computing the general term in the
large $ x $ expansion (which are convergent expansions) of $ W^{l}_{1}(x) $ for $ l=1,\ldots, 6 $ respectively, as
given by \eqref{resolvent_1} - \eqref{resolvent_6}.
\end{proof}
\

We conclude with an observation on the location of the zeros of the $ N $ coefficients that is satisfied by
all the cases that are accessible to us.
\begin{conjecture}
In addition to the palindromic/anti-palindromic property the numerator of the coefficients with respect to
$ N $ have simple zeros all lying on the unit circle, $ |\kappa|=1 $, and thus form complex conjugate pairs.
\end{conjecture}

\section{The special cases $\beta = 1, \, 2$ and 4}
\setcounter{equation}{0}

For simplicity in the final results for the moments and the exponential-type generating
functions\footnote{In contrast to our choice of the eigenvalue PDF in \S \ref{LoopEquations} given by \eqref{pdfSTAR}
and the potential $ V(\lambda) $} we will employ averages with respect to the G$\beta$E ensemble with the PDF
\eqref{GbetaEpdf}.
Thus we have $ m_{2p}(N,\kappa) := \left\langle \Tr G^{2p} \right\rangle_{{\rm G{\beta}E}} $, in comparison to
our earlier definition \eqref{resolvent}.
There are special orthogonal structures for $\beta = 1$, 2 and 4 which enables
special characterisations of the moments and the resolvent not available for general $\beta$. This structure
rests on the {\it semi-classical} character
of the underlying orthogonal polynomial system. For such a system the reproducing kernel is defined as
\begin{equation}
   K_{N}(x,x) = \sqrt{N}e^{-\frac{1}{2}x^2}\left[ p^{\prime}_{N}(x)p_{N-1}(x)-p_{N}(x)p^{\prime}_{N-1}(x) \right] ,
\label{OPSkernel}
\end{equation}
and the orthogonal polynomials $ \{p_n(x)\}^{\infty}_{n=0} $ normalised with respect to $ e^{-\frac{1}{2}x^2} $
are given by
\begin{equation}
   p_n(x) = \frac{1}{\sqrt{\sqrt{2\pi}2^n n!}}H_n\left(\frac{x}{\sqrt{2}}\right) =  \frac{1}{\sqrt{\sqrt{2\pi}n!}}He_n(x) ,
\label{normOP}
\end{equation}
where $ H_n(x) $, $ He_n(x) $ are the standard Hermite polynomials, see \S 18.3 of \cite{DLMF}. The density is
normalised so that
\begin{equation}
  \int^{\infty}_{-\infty} dx\;\rho_{(1)}(x) = N .
\end{equation}
The key relations we require are the generic three-term recurrence relation, which in our context is
\begin{equation}
   xp_{n}(x) = \sqrt{n+1}p_{n+1}(x) + \sqrt{n}p_{n-1}(x) ,
\label{OPS_3trr}
\end{equation}
the semi-classical property of the derivative
\begin{equation}
   p_{n}^{\prime}(x) := \frac{d}{dx}p_{n}(x) = \sqrt{n}p_{n-1}(x) ,
\label{OPS_D}
\end{equation}
and as a consequence the eigenvalue or second-order differential equation
\begin{equation}
  p_{n}^{\prime\prime}-xp_{n}^{\prime}+np_{n} = 0 .
\label{OPS_2ndODE}
\end{equation}

\subsection{$\kappa=1$ GUE Moments}
The density, as the one-point correlation function, has the classical evaluation \cite{rmt_Fo} of a determinant of
the reproducing kernel
\begin{equation}
   \rho_{(1)}(x) = \sqrt{\frac{N}{g}}\left. K_{N}(x,x) \right|_{x\mapsto \sqrt{\frac{N}{g}}x} ,
\label{GUE_densitykernel}
\end{equation}
where the kernel is given in \eqref{OPSkernel}.

A third order ordinary differential equation was found for the density and resolvent directly,
in the works of G{\"o}tze and Tikhomirov \cite{GT_2005} and Haagerup and Thorbj{\o}rnsen \cite{HT_2012}, however we will give an independent proof of
this fact from first principles.
\begin{theorem}[Lemma 2.1 of \cite{GT_2005}, Prop. 2.2 and Lemma 4.1 of \cite{HT_2012}]\label{GUEWode}
The resolvent $ W_{1}(x) $ satisfies the third order, inhomogeneous ordinary differential equation
\begin{equation}
    \frac{g^2}{N^2} W_{1}^{'''}+(4g-x^2) W_{1}^{'}+x W_{1} = 2N, \quad x \notin \mathbb{R} ,
\label{GUEresolventODE}
\end{equation}
subject to the boundary conditions, for fixed $ g, N $
\begin{equation}
    W_{1}(x) \mathop{\sim}_{x \rightarrow \infty} \frac{m_0}{x}+\frac{gm_2}{Nx^3}+\frac{g^2m_4}{N^2x^5}+\frac{g^3m_6}{N^3x^7}+\ldots ,
\label{resolventBC}
\end{equation}
and the moments are given by \eqref{GUE_moments}.
Furthermore, the density satisfies the homogeneous part of (\ref{GUEresolventODE}).
\end{theorem}
\begin{proof}
We will establish \eqref{GUEresolventODE} in a few steps, initially establishing that the density satisfies
the homogeneous form of \eqref{GUEresolventODE}. However we will work with the version where the independent
variable is not scaled for the bulk scaling regime purely for convenience and entailing no loss of generality,
which is just the version with $ g \mapsto N $
\begin{equation}
  \rho_{(1)}^{'''}+(4N-x^2)\rho_{(1)}^{'}+x\rho_{(1)} = 0 .
\label{GUE_densityODE}
\end{equation}
Thus $ \rho_{(1)} = K_{N}(x,x) $ and using \eqref{OPSkernel} we compute the first three derivatives with
respect to $ x $
\begin{align}
  \frac{1}{\sqrt{N}}e^{\frac{1}{2}x^2} K_{N} & = p_{N}^{\prime}p_{N-1}-p_{N}p_{N-1}^{\prime} ,
\\
  \frac{1}{\sqrt{N}}e^{\frac{1}{2}x^2} K_{N}^{\prime} & = -p_{N}p_{N-1} ,
\\
  \frac{1}{\sqrt{N}}e^{\frac{1}{2}x^2} K_{N}^{\prime\prime} & = xp_{N}p_{N-1}-p_{N}^{\prime}p_{N-1}-p_{N}p_{N-1}^{\prime} ,
\\
  \frac{1}{\sqrt{N}}e^{\frac{1}{2}x^2} K_{N}^{\prime\prime\prime} & = (4N-x^2)p_{N}p_{N-1}-xp_{N}^{\prime}p_{N-1}+xp_{N}p_{N-1}^{\prime} ,
\end{align}
where we have repeatedly used \eqref{OPS_2ndODE}. Furthermore, of the four bilinear products
$ p_{N}p_{N-1} $, $ p_{N}^{\prime}p_{N-1} $, $ p_{N}p_{N-1}^{\prime} $, $ p_{N}^{\prime}p_{N-1}^{\prime} $, only three are
independent as one can deduce
\begin{equation}
  p_{N}^{\prime}p_{N-1}^{\prime} = -N p_{N}p_{N-1}+x p_{N}^{\prime}p_{N-1} ,
\end{equation}
from \eqref{OPS_D} and \eqref{OPS_3trr}. Using the first three relations one can invert these for
$ p_{N}p_{N-1} $, $ p_{N}^{\prime}p_{N-1} $, $ p_{N}p_{N-1}^{\prime} $ as the determinant of the transformation is
non-vanishing. Thus we have
\begin{align}
  e^{-\frac{1}{2}x^2} p_{N}p_{N-1} & = -\frac{1}{\sqrt{N}} K_{N}^{\prime} ,
\\
  e^{-\frac{1}{2}x^2} p_{N}^{\prime}p_{N-1} & =  \frac{1}{2\sqrt{N}}\left[ K_{N}-xK_{N}^{\prime}-K_{N}^{\prime\prime} \right] ,
\\
  e^{-\frac{1}{2}x^2} p_{N}p_{N-1}^{\prime} & = -\frac{1}{2\sqrt{N}}\left[ K_{N}+xK_{N}^{\prime}+K_{N}^{\prime\prime} \right] .
\end{align}
Substituting these into the fourth relation gives \eqref{GUE_densityODE}. The inhomogeneous relation
now follows from the sequence of steps
\begin{multline}
  0 = \int^{\infty}_{-\infty}dx \frac{1}{z-x}\left[ \rho_{(1)}^{\prime\prime\prime}+(4N-x^2)\rho_{(1)}^{\prime}+x\rho_{(1)} \right]
    = \int^{\infty}_{-\infty}dx \frac{\rho_{(1)}^{'''}}{z-x}
\\
   +(4N-z^2)\int^{\infty}_{-\infty}dx \frac{\rho_{(1)}^{'}}{z-x} + \int^{\infty}_{-\infty}dx\, (z+x)\rho_{(1)}^{'}
     + z\int^{\infty}_{-\infty} dx \, \frac{\rho_{(1)}}{z-x} - \int^{\infty}_{-\infty} d x \, \rho_{(1)} .
\end{multline}
Now we integrate by parts the first three terms using $ \partial_x (z-x)^{-1} = -\partial_z (z-x)^{-1} $ and
assuming $ z \notin \mathbb{R} $, $ \rho_{(1)}(x) $, $ x\rho_{(1)}(x) $, $ \rho_{(1)}^{\prime}(x) $ and
$ \rho_{(1)}^{\prime\prime}(x) $ all vanish sufficiently rapidly as $ x \to \pm\infty $. This is justified
because our solution to \eqref{GUE_densityODE} is the single one, out of the three possible, that possesses exponential
decay at the boundaries. This also justifies our interchange of derivative and integral as the integrals
are uniformly and absolutely convergent. We find that the only boundary terms remaining on the right-hand side
are two copies of the normalisation integral, and thus \eqref{GUEresolventODE} follows once the bulk scaling
$ x\mapsto \sqrt{\frac{N}{g}}x $ is re-instated.
\end{proof}
\begin{remark}
As a consistency check we observe that the $ 1/N $ expansion of the resolvent \eqref{resolventN} along with
coefficients \eqref{resolvent_0}-\eqref{resolvent_6}, under the specialisation $ \kappa \to 1 $, identically satisfies
\eqref{GUEresolventODE} up to the error term of $ {\rm O}(N^{-8}) $.
\end{remark}

As one can see the third-derivative term in \eqref{GUEresolventODE} can be interpreted as a correction
term in the large $ N $ expansion, so this relation can serve to generate successive terms in such an
expansion.
An explicit large $ N $ expansion for the resolvent was found in Prop. 4.5 of \cite{HT_2012}, along with a
recurrence for the coefficients. Let
\begin{equation}
   \frac{g}{N}W_1(x) = \eta_0(x)+\frac{\eta_1(x)}{N^2}+ \ldots +\frac{\eta_{k}(x)}{N^{2k}}+{\rm O}(N^{-2k-2}),
\quad k \in \mathbb{N} ,
\label{GUEresolventEXP}
\end{equation}
where
\begin{equation}
  \eta_{0}(x) = \tfrac{1}{2}\left[ x-\sqrt{x^2-4g} \right] ,
\end{equation}
and
\begin{equation}
   \eta_{j}(x) = \sum^{3j-1}_{r=2j} C_{j,r}(x^2-4g)^{-r-1/2} , \quad j \in \mathbb{N} .
\label{GUEeta}
\end{equation}
Then for $ 2j+2\leq r \leq 3j+2 $ we have
\begin{equation}
   C_{j+1,r} = g^2\frac{(2r-3)(2r-1)}{r+1}\left[ (r-1)C_{j,r-2}+g(4r-10)C_{j,r-3} \right] ,
\label{GUErecurC}
\end{equation}
and $ C_{j,2j-1} = C_{j,3j} = 0 $.

In Haagerup and Thorbj{\o}rnsen \cite{HT_2003} and Ledoux's works \cite{Le_2009, Le_2004} a linear ordinary
differential equation was derived for the exponential generating function
(\ref{EXPGF}).
In Haagerup and Thorbj{\o}rnsen's \cite{HT_2003} notation we have $ c(p,N) = m_{2p}(N,1) $ whilst in Ledoux's
notation \cite{Le_2009} $ a_{p}^{N} \equiv m_{2p}(N,1) $.
We give an alternative proof of this result using the result of Theorem \ref{GUEWode}.
\begin{theorem}[Eqs. (0.2, 2.16) of \cite{HT_2003}, Eq. (18) of \cite{Le_2009}]\label{GUEu}
The GUE exponential generating function $ u(t,N) $ satisfies the differential equation
\begin{equation}
  t u^{\prime \prime} + 3u^{\prime} - t(t^{2} + 4N)u =0 ,
\label{GUE_uODE}
\end{equation}
subject to the boundary conditions
\begin{equation}
  u(t,N) \mathop{\sim}_{t \rightarrow 0 } N + \frac{N^{2}}{2!} t^{2} + \cdots .
\label{GUE_bcODE}
\end{equation}
The solution defined above is
\begin{equation}
  u(t,N) = N e^{-t^{2}/2}{}_{1}F_{1}(1+N,2;t^{2}) ,
\label{GUE_solnODE}
\end{equation}
where $ {}_{1}F_{1}(1+N,2;t^{2}) $ is the regular confluent hypergeometric function \cite{DLMF}.
\end{theorem}
\begin{proof}
We start with the Laplace transform \eqref{laplaceXfm} and compute the difference of the left-hand side
and right-hand sides of \eqref{GUEresolventODE}, yielding
\begin{equation}
 0 = \sqrt{\frac{N}{g}}\int^{\infty}_{0} dt\; u e^{-\sqrt{\frac{N}{g}}xt} \left[ -\sqrt{\frac{g}{N}}t^3-(4g-x^2)\sqrt{\frac{N}{g}}t+x \right] - 2N .
\end{equation}
Now we employ the identity $ -\sqrt{\frac{N}{g}}x e^{-\sqrt{\frac{N}{g}}xt} = \partial_t e^{-\sqrt{\frac{N}{g}}xt} $
and for higher orders, and integrate by parts which gives us
\begin{equation}
 0 = \int^{\infty}_{0} dt\; e^{-\sqrt{\frac{N}{g}}xt} \left[ (tu)^{''}+u^{'}-t^3u-4Ntu \right]
     + \left[ \left( -\sqrt{\frac{N}{g}}tu-2u-tu^{'} \right)e^{-\sqrt{\frac{N}{g}}xt} \right]^{\infty}_{0} - 2N .
\end{equation}
Assuming the upper limit vanishes for each of the three terms, in some sector $ |{\rm arg}(t)| < \pi $, then
the evaluations $ u(0) = N $, $ u^{'}(0) = 0 $ lead to the cancellation of the inhomogeneous terms.
Thus we have \eqref{GUE_uODE}. In fact \eqref{GUE_uODE}, after removing the factor $ e^{-t^{2}/2} $, is
one of the standard forms of the confluent hypergeometric differential equation, see \S 13.2 of \cite{DLMF},
and only the regular part is admissible because the other solution, $ U(N+1,2,t^2) \sim t^{-2}/N! $
as $ t \to 0 $ and has $ \log(t) $ terms.
\end{proof}

\begin{remark}
As $t \rightarrow +\infty$ with $ |{\rm Arg}(t^2)| \leq \frac{\pi}{2}-\delta $ and $ \delta>0 $
\begin{equation}
   u(t) \sim N  e^{-t^{2}/2} \cdot  e^{t^{2}} \frac{t^{N-1}}{N!} =\frac{t^{N-1}}{(N-1)!} e^{t^{2}/2} ,
\label{GUE_uEXP}
\end{equation}
which implies that the Laplace transform of $ u(t) $ does not exist, unless the integral is taken along a ray
such that $ {\rm Re}(t^2)<0 $.
\end{remark}

A direct consequence of the Theorems \ref{GUEWode} or \ref{GUEu} is that the moments satisfy a linear recurrence relation.
We can give an independent derivation of such a recurrence relation.
\begin{theorem}[\cite{HZ_1986}, Theorem 4.1 of \cite{HT_2003}, Theorem 1 of \cite{Le_2009}]\label{GUErecurrence}
The moments of the GUE satisfy the linear difference equation
\begin{gather}
(p+1)m_{2p} = (4p-2) N m_{2p-2} + (p-1)(2p-1)(2p-3) m_{2p-4} ,
\label{GUE_RR}
\end{gather}
subject to the initial conditions $ m_{0}=N,\ m_{2}=N^{2} $.
\end{theorem}
\begin{proof}
Taking the homogeneous form of \eqref{GUEresolventODE} for the density and integrating it against the monomial
$ x^{2p-1} $ we deduce using integration by parts
\begin{align*}
 0 =\; & \int^{\infty}_{-\infty} dx\; x^{2p-1} \left[ x \rho_{(1)}+(4g-x^2) \rho_{(1)}^{'}+\frac{g^2}{N^2} \rho_{(1)}^{'''} \right]
\\
   =\; & m^{*}_{2p} + 4g\left\{ \left[ x^{2p-1}\rho_{(1)} \right]^{\infty}_{-\infty} - (2p-1)\int^{\infty}_{-\infty} dx\;  x^{2p-2}\rho_{(1)} \right\}
\\
     & -\left\{ \left[ x^{2p+1}\rho_{(1)} \right]^{\infty}_{-\infty} - (2p+1)\int^{\infty}_{-\infty} dx\;  x^{2p}\rho_{(1)} \right\}
\\
     & \phantom{m^{*}_{2p} +}
                  + \frac{g^2}{N^2}\left\{ \left[ x^{2p-1}\rho^{''}_{(1)} \right]^{\infty}_{-\infty} - (2p-1)\int^{\infty}_{-\infty} dx\;  x^{2p-2}\rho^{''}_{(1)} \right\} .
\end{align*}
Further use of integration by parts shows
\begin{align*}
 0  = & m^{*}_{2p}-4g(2p-1)m^{*}_{2p-2}+(2p+1)m^{*}_{2p}-\frac{g^2}{N^2}(2p-1)(2p-2)(2p-3) m^{*}_{2p-4}
\\
     & \phantom{m^{*}_{2p}}
                  + \Big[ 4gx^{2p-1}\rho^{}_{(1)}-x^{2p+1}\rho^{}_{(1)}
\\
     & \phantom{m^{*}_{2p} +}
                  +\frac{g^2}{N^2}\left( x^{2p-1}\rho^{''}_{(1)}-(2p-1)x^{2p-2}\rho^{'}_{(1)}+(2p-1)(2p-2)x^{2p-3}\rho^{}_{(1)} \right) \Big]^{\infty}_{-\infty} .
\end{align*}
Now we require our density to satisfy $ x^{2p+1}\rho^{}_{(1)} $, $ x^{2p-2}\rho^{'}_{(1)} $, $ x^{2p-1}\rho^{''}_{(1)} \to 0 $
as $ x\to \pm \infty $ for all $ p \in \mathbb{N}  $, which our solution indeed satisfies.
Adjusting for \eqref{Mscaling} we have \eqref{GUE_RR}.
\end{proof}

Using \eqref{GUE_RR} one can efficiently generate the low-order moments and we record the first seven for
checking purposes
\begin{equation}
\begin{aligned}\label{GUE_moments}
m_{0}  &= N ,
\\
m_{2}  &= N^{2} ,
\\
m_{4}  &= N(2N^{2}+1) ,
\\
m_{6}  &= 5N^{2}(N^{2} + 2) ,
\\
m_{8}  &= 7N(2N^{4} + 10N^{2} + 3) ,
\\
m_{10} &= 21N^{2}(2N^{4}+20N^{2}+23) ,
\\
m_{12} &= 33 N (4N^{6}+70N^{4}+196N^{2}+45) .
\end{aligned}
\end{equation}

Ledoux \cite{Le_2009} has utilised this recurrence relation to derive the large $ N $ behaviour of the moments
up to the first non-zero correction. It is easy to extend this method to obtain more corrections.
\begin{theorem}\label{GUE_largeNmoment}
As $ N \rightarrow \infty $ for fixed $ p $, $ m_{2p}(N,1) = O(N^{p+1})$ . Furthermore
\begin{multline}
  \frac{m_{2p}(N,1)}{C_{p}N^{p+1}} = 1 + \frac{1}{12}(p+1)p(p-1)N^{-2} + \frac{1}{1440}(p+1)p(p-1)(p-2)(p-3)(5p-2)N^{-4}
\\
   + \frac{1}{362880}(p+1)p(p-1)(p-2)(p-3)(p-4)(p-5)(35p^2-77p+12)N^{-6} + O(N^{-8}) .
\label{GUE_expN}
\end{multline}
\end{theorem}
\begin{proof}
We proceed by peeling off successive terms in the large $ N $ development of
\begin{equation*}
  m_{2p} = C_p N^{p+1}\left[ 1 + a_{p}N^{-2}+b_{p}N^{-4}+c_{p}N^{-6} + \ldots \right] ,
\end{equation*}
the odd orders are not present as one can see from the substitution
$ m_{2p} = \frac{(2p)!}{p!(p+1)!}N^{p+1}X_{p} $ which yields the difference equation
$ X_{p}=X_{p-1}+\frac{p(p-1)}{4N^2}X_{p-2} $. At the $ N^{p-1} $ order we find the first order
inhomogeneous difference equation $ p-p^2-4a_{p-1}+4a_{p} = 0 $ which is solved with
$ a_{1} = 0 $ to give $  a_{p} = \frac{1}{12}(p^3-p) $. Using this solution we
find at the $ N^{p-3} $ order another first order
inhomogeneous difference equation $ p(p-1)^2(p-2)(p-3)+48b_{p-1}-48b_{p} = 0 $
subject to $ b_{3} = 0 $. Its solution is $ b_{p} = \frac{1}{1440}(p+1)p(p-1)(p-2)(p-3)(5p-2) $.
Again using these preceding solutions we have at the $ N^{p-5} $ order
$ -p(p-1)^2(p-2)(p-3)(p-4)(p-5)(5p-12)-5760c_{p-1}+5760c_{p} = 0 $ with $ c_{5} = 0 $. The
solution is $ c_{p} =\frac{1}{362880}(p+1)p(p-1)(p-2)(p-3)(p-4)(p-5)(35p^2-77p+12) $.
\end{proof}

\begin{remark} This agrees with Theorem \ref{T19}  in the case $ h=0 $, $ \kappa=1 $.
\end{remark}

An explicit formula for the moments is given in Mehta \cite{Mehta_3rd}, see \S 6.5.6, and also by Mezzadri
and Simm 2011 \cite{MS_2011}, see Theorem 2.9,
Equations (31) and (32). In their notations our moment can be written $ m_{2p}(N,1) \equiv C(p,N) $ and
$ m_{2p}(N,1) \equiv 2^{p}M_{G}^{(2)} (2p,N) $.
\begin{theorem}[\cite{Mehta_3rd},\cite{MS_2011}]
For all $ N>0 $, $ p \in \mathbb{Z}_{\geq 0} $ the GUE moments are given by Mehta's evaluation
\begin{equation}
  m_{2p}(N,1) = \frac{(2p)!}{2^{p}p!}N{}_{2}F_{1}(-p,1-N;2;2)
              = \frac{(2p)!}{2^{p}p!}\sum^{p}_{j=0}{p \choose j}{N \choose j+1}2^j ,
\label{GUEevaluation_mehta}
\end{equation}
or by Mezzadri and Simm's
\begin{equation}
  m_{2p}(N,1) =
  \begin{cases} \displaystyle
        \frac{2^{N+p} \Gamma\left( \frac{N}{2} + 1 \right) \Gamma\left( \frac{N}{2} \right) }{\pi^{1/2} (2p+1) \Gamma(N)  }
        \sum_{j=0}^{\min \left( \frac{N}{2}-1,p \right) } { p \choose j } { p+1 \choose j +1}\left( \frac{N}{2} - j \right)_{p + \frac{1}{2} }, & N \mbox{ even,} \\
  \\ \displaystyle
        \frac{2^{N+p} \Gamma\left( \frac{N+1}{2} \right)^{2} }{\pi^{1/2} (2p+1) \Gamma(N)  }
        \sum_{j=0}^{\min \left( \frac{N-1}{2},p \right) } { p \choose j } { p+1 \choose j }
        \left( \frac{N+1}{2}-j \right)_{p+\frac{1}{2} },  & N \mbox{ odd.}
  \end{cases}
\label{GUE_finitesum}
\end{equation}
\end{theorem}

\begin{remark}
In a well-known work Harer and Zagier \cite{HZ_1986} found a simple result for the
generating function (\ref{subEXPGF}) of the GUE moments
\begin{equation}
   \phi(s,N) = \frac{1}{2s^2}\left[ \left(\frac{1+s^2}{1-s^2}\right)^{N}-1 \right] ,
\end{equation}
which follows directly from the definition \eqref{subEXPGF} and the evaluation \eqref{GUEevaluation_mehta}. The
latter explicit formula agrees with the specialisation of $ \kappa=1 $ in the cases $ 0\leq p\leq 6 $ of \eqref{moment0}-\eqref{moment12}.
\end{remark}

\subsection{$\kappa = 1/2$ GOE Moments}
We revise the well-known explicit result for the density of eigenvalues in the orthogonal Gaussian ensemble,
as given in \S 4, p. 158, 9  of \cite{AFNvM_2000} (after correcting for the typographical error), and adapted
to our slightly differing conventions. From this work we deduce
\begin{multline}
   \rho_{(1)}(x,N) = \sqrt{\frac{N}{g}} \left\{
    K_{N}(x,x) + \frac{e^{-\frac{1}{4}x^2}H_{N-1}(2^{-1/2}x)}{\sqrt{\pi}2^{N+2}(N-1)!}\int^{\infty}_{-\infty}dt\;{\rm sgn}(x-t)e^{-\frac{1}{4}t^2}H_{N}(2^{-1/2}t)
                                        \right.
\\                                      \left.
               + \chi_{N\in 2\mathbb{N}+1}\frac{e^{-\frac{1}{4}x^2}H_{N-1}(2^{-1/2}x)}{\int^{\infty}_{-\infty}dt\;e^{-\frac{1}{4}t^2}H_{N-1}(2^{-1/2}t)}
                                        \right\}_{x\mapsto \sqrt{\frac{N}{g}}x} ,
\label{GOE_density}
\end{multline}
where $ K_{N}(\cdot,\cdot) $ is given by the kernel \eqref{OPSkernel}.

The orthogonal analogue of \eqref{GUEresolventODE} is a fifth order ordinary differential equation for the resolvent.
\begin{theorem}\label{GOEWode}
The resolvent $ W_{1}(x) $ for the GOE satisfies the fifth order, linear inhomogeneous ordinary differential equation
\begin{multline}
  -4\frac{g^4}{N^4}W_{1}^{(V)}+5\left[ x^2-(4N-2)\frac{g}{N} \right]\frac{g^2}{N^2}W_{1}^{\prime\prime\prime}-6\frac{g^2}{N^2}xW_{1}^{\prime\prime}
\\
  +\left[ -x^4+(8N-4)\frac{g}{N}x^2+(-16N^2+16N+2)\frac{g^2}{N^2} \right]W_{1}^{\prime}+x\left[ x^2-(4N-2)\frac{g}{N} \right]W_{1}
\\
  = 2 N (x^2-4 g)+10g ,
\label{GOEresolventODE}
\end{multline}
subject to the boundary conditions \eqref{resolventBC}, for fixed $ g, N $, and the moments are given by \eqref{GOE_moments}.
\end{theorem}
\begin{proof}
Our proof will be a natural extension of the methods adopted in the proof of Theorem \ref{GUEWode}.
As in the proof of that Theorem we will establish the result for the unscaled system ($ g\mapsto N $) to simplify
matters. Thus we first recast \eqref{GOE_density} in terms of the orthonormal polynomials $ \{p_n\}_{n=0}^{\infty} $
\begin{equation}
  \frac{1}{\sqrt{N}}\rho_{(1)} = e^{-\frac{1}{2}x^2} \left[ p_{N}^{\prime}p_{N-1}-p_{N}p_{N-1}^{\prime} \right]
   + e^{-\frac{1}{4}x^2}\left[ \tfrac{1}{2}q_{N}+A_{N} \right]p_{N-1} ,
\end{equation}
where the constant $ A_{n} $ is defined as
\begin{equation}
 A_{n} = -\frac{1}{4}\chi_{n\in 2\mathbb{Z}}D_{n}+\frac{1}{\sqrt{n}}\chi_{n\in 2\mathbb{Z}+1}D^{-1}_{n-1}, \quad
 D_{n} = \pi^{1/4}2^{3/4}\sqrt{\frac{(n-1)!!}{n!!}} .
\end{equation}
The new variable $ q_{n} $ is defined as
\begin{equation}
  q_{n}(x) := \int^{x}_{-\infty} dt\; e^{-\frac{1}{4}t^2}p_{n}(t) ,
\end{equation}
and clearly $ q^{\prime}_{n} = e^{-\frac{1}{4}x^2}p_{n}(x) $. The basis of bilinear products is now five dimensional
with basis
\begin{equation*}
    \{ p_{N}p_{N-1},p_{N}^{\prime}p_{N-1},p_{N}p_{N-1}^{\prime},[\tfrac{1}{2}q_{N}+A_{N}]p_{N-1},[\tfrac{1}{2}q_{N}+A_{N}]p_{N-1}^{\prime} \} .
\end{equation*}
Using the relation for $ q^{\prime}_{n} $ and \eqref{OPS_2ndODE} we compute the first four derivatives of $ \rho_{(1)} $
which we write in matrix form
\begin{multline}
  \frac{1}{\sqrt{N}}
  \begin{pmatrix}
    \rho_{(1)} \\ \rho_{(1)}^{\prime} \\ \rho_{(1)}^{\prime\prime} \\ \rho_{(1)}^{\prime\prime\prime} \\ \rho_{(1)}^{(IV)} \\
  \end{pmatrix} =
\\
  \begin{pmatrix}
   0 & 1 & -1 & 1 & 0 \\
   -\tfrac{1}{2} & 0 & 0 & -\tfrac{1}{2}x & 1 \\
   \tfrac{1}{4}x & -\tfrac{1}{2} & 0 & \tfrac{1}{4}x^2-N+\tfrac{1}{2} & 0 \\
   -\tfrac{1}{8}x^2+\tfrac{1}{2}N+\tfrac{1}{2} & -\tfrac{1}{4}x & \tfrac{1}{4}x & -\tfrac{1}{8}x^3+\tfrac{1}{2}(N+\tfrac{1}{2})x & \tfrac{1}{4}x^2-N+\tfrac{1}{2} \\
   \tfrac{1}{16}x^3-\tfrac{1}{4}(N+\tfrac{3}{2})x & -\tfrac{1}{8}x^2+\tfrac{1}{2}N+\tfrac{1}{4} & 1 & \tfrac{1}{16}x^4-\tfrac{1}{2}(N+\tfrac{1}{2})x^2+N^2-N+\tfrac{3}{4} & x
  \end{pmatrix}
\\ \cdot
  \begin{pmatrix}
   e^{-\frac{1}{2}x^2}p_{N}p_{N-1} \\ e^{-\frac{1}{2}x^2}p_{N}^{\prime}p_{N-1} \\ e^{-\frac{1}{2}x^2}p_{N}p_{N-1}^{\prime} \\ e^{-\frac{1}{4}x^2}[\tfrac{1}{2}q_{N}+A_{N}]p_{N-1} \\ e^{-\frac{1}{4}x^2}[\tfrac{1}{2}q_{N}+A_{N}]p_{N-1}^{\prime}
  \end{pmatrix} .
\end{multline}
This is invertible for $ N>1 $ as the determinant is $ -\tfrac{9}{8}(N-1) $. The fifth derivative is
\begin{multline}
 \frac{1}{\sqrt{N}}\rho_{(1)}^{(V)} =
  \left[ -\tfrac{1}{32}x^4+\tfrac{1}{4}(N+\tfrac{7}{4})x^2-\tfrac{1}{2}N^2-\tfrac{13}{4}N+1 \right]  e^{-\frac{1}{2}x^2}p_{N}p_{N-1}
\\
  + \left[ -\tfrac{1}{16}x^3+\tfrac{1}{4}(N+\tfrac{5}{2})x \right] e^{-\frac{1}{2}x^2}p_{N}^{\prime}p_{N-1}
  + \left[ \tfrac{1}{16}x^3-\tfrac{1}{4}(N-\tfrac{1}{2})x \right]  e^{-\frac{1}{2}x^2}p_{N}p_{N-1}^{\prime}
\\
  + \left[ -\tfrac{1}{32}x^5+\tfrac{1}{4}(N+\tfrac{3}{2})x^3+\tfrac{1}{2}(-N^2-3N+\tfrac{1}{4})x \right]  e^{-\frac{1}{4}x^2}[\tfrac{1}{2}q_{N}+A_{N}]p_{N-1}
\\
  + \left[ \tfrac{1}{16}x^4+\tfrac{1}{2}(-N+\tfrac{1}{2})x^2+N^2-N+\tfrac{7}{4} \right]  e^{-\frac{1}{4}x^2}[\tfrac{1}{2}q_{N}+A_{N}]p_{N-1}^{\prime} .
\end{multline}
Substituting the solution for the bilinear products into this expression we get
\begin{equation}
 -4\rho_{(1)}^{(V)}+5(x^2-4N+2)\rho_{(1)}^{\prime\prime\prime}-6x\rho_{(1)}^{\prime\prime}+[-x^4+(8N-4)x^2-16N^2+16N+2]\rho_{(1)}^{\prime}+x(x^2-4N+2)\rho_{(1)} = 0 .
\end{equation}
Restoring the bulk scaling $ x\mapsto \sqrt{\frac{N}{g}}x $ we have the homogeneous form of \eqref{GOEresolventODE}.
To find the differential equation for $ W_1 $ we repeat the methods employed in the proof of Theorem
\ref{GUEWode}, except that there are more terms to treat. Integrating the homogeneous form of the differential
equation against $ (z-x)^{-1} $ on $ \mathbb{R} $, and performing the subtractions for the $x$-dependent
coefficients we arrive at
\begin{multline}
 0 =  -4\frac{g^4}{N^4}\int^{\infty}_{-\infty}dx\;\frac{\rho_{(1)}^{(V)}}{(z-x)}
       +5\left[ z^2-(4N-2)\frac{g}{N} \right]\frac{g^2}{N^2}\int^{\infty}_{-\infty}dx\;\frac{\rho_{(1)}^{\prime\prime\prime}}{(z-x)}
\\
     -6\frac{g^2}{N^2}z\int^{\infty}_{-\infty}dx\;\frac{\rho_{(1)}^{\prime\prime}}{(z-x)}
      +\left[ -z^4+(8N-4)\frac{g}{N}z^2+(-16N^2+16N+2)\frac{g^2}{N^2} \right]\int^{\infty}_{-\infty}dx\;\frac{\rho_{(1)}^{\prime}}{(z-x)}
\\
       +z\left[ z^2-(4N-2)\frac{g}{N} \right]\int^{\infty}_{-\infty}dx\;\frac{\rho_{(1)}}{(z-x)}
        -5\frac{g^2}{N^2}\int^{\infty}_{-\infty}dx\;(x+z)\rho_{(1)}^{\prime\prime\prime}
\\
       +6\frac{g^2}{N^2}\int^{\infty}_{-\infty}dx\;\rho_{(1)}^{\prime\prime}
        +\int^{\infty}_{-\infty}dx\;[z^3+z^2x+zx^2+x^3-(8N-4)\frac{g}{N}(x+z)]\rho_{(1)}^{\prime}
\\
         +\int^{\infty}_{-\infty}dx\;[-z^2-zx-x^2+(4N-2)\frac{g}{N}]\rho_{(1)} .
\end{multline}
Making similar observations on $ \rho_{(1)} $ concerning its decay as $ x\to \pm\infty $ as we did in the
proof of Theorem \ref{GUEWode} we can conclude
$ \int dx\, (z-x)^{-1}\partial^n_x\rho_{(1)}(x) = \partial^n_z W_1(z) $ for $ 0 \leq n \leq 5 $. Of the four
final terms of the above expression only a few are non-zero and these contribute
$ [ -2z^2+(12N-6)\frac{\displaystyle g}{N}]m_{0}-4\frac{\displaystyle g}{N}m_{2} $.
From the knowledge of the first two moments we deduce the inhomogeneous term and arrive at (\ref{GOEresolventODE}).
\end{proof}
\begin{remark}
As a check we can verify that the $ 1/N $ expansion of the resolvent \eqref{resolventN} along with
coefficients \eqref{resolvent_0}-\eqref{resolvent_6}, under the specialisation $ \kappa \to 1/2 $, identically satisfies
\eqref{GOEresolventODE} up to the error term of $ {\rm O}(N^{-7}) $.
\end{remark}

In \cite{Le_2009} a linear ordinary differential equation was derived for the exponential generating function,
however we can easily re-derive this from the preceding theorem. Here Ledoux's definitions imply $ b_{p}^{N} = m_{2p}(N,1/2) $.
\begin{theorem}[\cite{Le_2009}, Equation (27)]\label{GOEu}
The GOE exponential generating function $ u(t,N) $ satisfies the fourth order linear ordinary differential equation
\begin{multline}
t u^{{\rm (IV)}} + 5 u^{\prime\prime\prime} - t( 5t^{2} + 8N -4) u^{(\prime\prime)}
 - (36 t^{2} + 20N -10) u^{(\prime)} \\
+ t \left[ 4t^{4} + (20N -10)t^{2} + 16N^{2} -16N - 44	\right] u =0 ,
\label{GOE_uODE}
\end{multline}
or equivalently if we define $ U \equiv u^{\prime \prime} - (4t^{2} + 4N -2) u $ then $ U(t) $ satisfies
\begin{equation}
  t U^{\prime  \prime} + 5U^{\prime } -t(t^{2} + 4N -2)U = 0 .
\label{GOE_odeU}
\end{equation}
The solutions are subject to the boundary conditions
\begin{equation}
    u(t) \mathop{\sim}_{t \rightarrow 0 } N + \frac{N(N+1)}{2!} t^{2} + \frac{N(2N^2+5N+5)}{4!} t^{4} + \frac{N(5N^3+22N^2+52N+41)}{6!} t^{6} \cdots .
\label{GOE_bcODE}
\end{equation}
\end{theorem}
\begin{proof}
We utilise the same method as given in the proof of Theorem \ref{GUEu}. In our intermediate step we find
\begin{multline}
 0 = \frac{g}{N} \int^{\infty}_{0}dt\, e^{-\sqrt{\frac{N}{g}}xt}\Big\{ \partial^4_{t}(tu)+\partial^3_{t}u-\partial^2_{t}[5t^3u+(8N-4)tu]-\partial_{t}[6t^2u+(4N-2)u]
\\
     +[4t^5+5(4N-2)t^3-(-16N^2+16N+2)t]u \Big\}
\\
     + \frac{g}{N} \left[ e^{-\sqrt{\frac{N}{g}}xt}\left( -\partial^3_{t}(tu)-\partial^2_{t}u+\partial_{t}(5t^3u+(8N-4)tu)+(6t^2+4N-2)u \right)
                   \right.
\\                 \left.
                      + \partial_{t}e^{-\sqrt{\frac{N}{g}}xt} \left( \partial^2_{t}(tu)+\partial_{t}u-(5t^3u+(8N-4)tu \right)
                      + \partial^2_{t}e^{-\sqrt{\frac{N}{g}}xt}\left( -\partial_{t}(tu)-u \right) + \partial^3_{t}e^{-\sqrt{\frac{N}{g}}xt}(tu)
                   \right]^{\infty}_{0}
\\   -2N(x^2-4g)-10 g.
\end{multline}
Again assuming the upper terminal contribution of the boundary term vanish we compute the lower terminal to be
$ \frac{\displaystyle g}{N}\left(4u^{\prime\prime}(0)-(12N-6)u(0)+2\frac{\displaystyle N}{g}x^2u(0) \right) $.
Using the data for $ m_0, m_2 $ this cancels the inhomogeneous terms from the original differential equation and
we have \eqref{GOE_uODE}.
\end{proof}

Ledoux has also shown that a linear recurrence relation for the moments follows from the above result, which can also be derived
directly from Theorem \ref{GOEWode}.
\begin{theorem}[\cite{Le_2009}, Theorem 2]\label{GOErecurrence}
The GOE moments $ m_{p}(N,1/2) $ satisfy the fourth order, linear difference equation
\begin{multline}
  (p+1)m_{2p} = (4p-1)(2N-1) m_{2p-2} + (2p-3)( 10p^{2} -9p - 8N^{2} + 8N) m_{2p-4}
\\
              - 5(2p-3)(2p-4)(2p-5)(2N-1) m_{2p-6} - 2(2p-3)(2p-4)(2p-5)(2p-6)(2p-7) m_{2p-8} ,
\label{GOE_RR}
\end{multline}
subject to the initial values $ m_{0}=N, \ m_{2}=N(N+1) $ from $ p\geq 2 $.
\end{theorem}
\begin{proof}
As in the proof of Theorem \ref{GUErecurrence} we integrate $ x^{2p-3} $ against the homogeneous form of
\eqref{GOEresolventODE} for $ \rho_{(1)} $ on $ \mathbb{R} $, and after integrating by parts we find
\begin{multline}
 0 = \left(\frac{g}{N}\right)^p
     \Big\{
     (2p+2)m_{2p}-(4N-2)(4p-1)m_{2p-2}
\\
     -(2p-3)[-16N^2+16N+2+6(2p-2)+5(2p-1)(2p-2)]m_{2p-4}
\\
      +5(4N-2)(2p-3)(2p-4)(2p-5)m_{2p-6}
\\
      +4(2p-3)(2p-4)(2p-5)(2p-6)(2p-7)m_{2p-8} \Big\}
\\
      + \text{boundary terms containing $ x^{2p+1}\rho_{(1)}, \ldots ,x^{2p-3}\rho_{(1)}^{(IV)} $ as $ x \to \pm\infty $} .
\end{multline}
Clearly $ x^{2p+1}\rho_{(1)} $, $ x^{2p-2}\rho_{(1)}^{\prime} $, $ x^{2p-1}\rho_{(1)}^{\prime\prime} $, $ x^{2p-4}\rho_{(1)}^{\prime\prime\prime} $,
$ x^{2p-3}\rho_{(1)}^{(IV)} $ all vanish exponentially fast as $ x \to \pm\infty $ for all $ p\geq 2 $ and we are justified in
neglecting the boundary terms. Eq. \eqref{GOE_RR} then follows.
\end{proof}
This recurrence relation is an efficient way to generate low order moments, of which we list the first seven for
checking purposes
\begin{equation}
\begin{aligned}\label{GOE_moments}
m_{0} &= N ,
\\
m_{2} &= N^{2} + N ,
\\
m_{4} &= 2N^{3} + 5N^{2} + 5N ,
\\
m_{6} &= 5N^{4} + 22N^{3} + 52N^{2} + 41N ,
\\
m_{8} &= 14N^{5} + 93N^{4} + 374N^{3} + 690N^{2} + 509N ,
\\
m_{10} &= 42N^{6} + 386N^{5} + 2290N^{4} + 7150N^{3} + 12143N^{2} + 8229N ,
\\
m_{12} &= 132 N^{7} + 1586 N^{6} + 12798 N^{5} + 58760 N^{4} + 167148 N^{3} + 258479 N^{2} + 166377 N .
\end{aligned}
\end{equation}

Again the recurrence relation enables one to compute the large $ N $ corrections to the moments and Ledoux has
given the first correction beyond the leading order in \cite{Le_2009}. We require more terms beyond the first
correction, and these can be easily found using the recurrence relation.
\begin{theorem}\label{GOEfirst}
As $N \rightarrow \infty$ for fixed $ p $ then $ m_{2p} = O(N^{p+1}) $ and the sub-leading coefficients are given by
\begin{multline}
   \frac{m_{2p}(N,\frac{1}{2})}{N^{p+1}} = C_{p}
    + 2^{2p-1} \left[ 1-\frac{\Gamma(p+\frac{1}{2})}{\sqrt{\pi}\Gamma(p+1)} \right]N^{-1}
\\
    + \frac{1}{3}4^{p-1}p \left[ -3+\frac{(7p-1) \Gamma(p+\frac{1}{2})}{\sqrt{\pi}\Gamma(p+1)} \right]N^{-2}
\\
    + \frac{1}{3}4^{p-2}p(p-1) \left[ 8p-7-\frac{(14p-4)\Gamma(p+\frac{1}{2})}{\sqrt{\pi}\Gamma(p+1)} \right]N^{-3}
\\
    + \frac{1}{45}2^{2p-5}p(p-1)(p-2) \left[ -15(8p-9)+\frac{(185p^2-317p+6)\Gamma(p+\frac{1}{2})}{\sqrt{\pi}\Gamma(p+1)} \right]N^{-4}
\\
    + \frac{1}{45}4^{p-4}p(p-1)(p-2)(p-3) \left[  320p^2-1008p+487-\frac{4(185p^2-387p+28) \Gamma(p+\frac{1}{2})}{\sqrt{\pi}\Gamma(p+1)} \right]N^{-5}
\\
    + \frac{1}{2835}2^{2p-9} p(p-1)(p-2)(p-3)(p-4) \left[ -63(320p^2-1168p+675) \phantom{\frac{\Gamma(p+\frac{1}{2})}{\sqrt{\pi}\Gamma(p+1)}} \right.
\\ \left.
    +\frac{4(6209p^3-29106p^2+26605p-60) \Gamma(p+\frac{1}{2})}{\sqrt{\pi}\Gamma(p+1)} \right]N^{-6}
    + {\rm O}(N^{-7}) .
\label{GOE_expN}
\end{multline}
\end{theorem}
\begin{proof}
This is derived using the same methods as given in the proof of Theorem \ref{GUE_largeNmoment}. There are
three practical differences with the GUE case - the odd orders are present in addition the even ones, that the
inhomogeneous difference equations are now of second order and the inhomogeneous terms involve Gamma functions.
\end{proof}

\begin{remark}
This agrees with Theorem \ref{T19} in the case $ \kappa = 1/2 $.
\end{remark}

In Theorem 4.2 of \cite{GJ_1997} Goulden and Jackson derived an explicit formula for the GOE moments.
In addition another formula for these moments has been deduced by Mezzadri and Simm 2011 \cite{MS_2011}, see Equation (34).
Their notation is related to ours by $ m_{2p}(N,1/2) \equiv 2^{p}M_{G}^{(1)}(2p,N) $.
\begin{theorem}[\cite{GJ_1997}, \cite{MS_2011}]
For all $ p, N $ the GOE moments are
\begin{equation}
  m_{2p}(N,1/2) = m_{2p}(N-1,1) + p!\sum^{p}_{i=0}2^{2p-i}\sum^{p}_{j=0}{p-\tfrac{1}{2} \choose p-j}{i+j-1 \choose i}{\frac{N-1}{2} \choose j} .
\end{equation}
For $ N $ even the GOE moments were given by Mezzadri and Simm as
\begin{multline}
m_{2p}(N,1/2) = m_{2p}(N-1,1)
\\
    - 2^{p}\sum_{j=1}^{\min \left( \frac{N}{2} -1 , p \right) }  \sum_{i=0}^{\min \left(p, \frac{N}{2} -1 -j  \right) }{ p \choose i } { p \choose i+j}
  \frac{ \left( \frac{N}{2}-i-j \right)_{p+\frac{1}{2}} }{ \left( \frac{N}{2}-j \right)_{\frac{1}{2}}  }
+ \phi_{p}(N) ,
\label{GOE_explicit}
\end{multline}
where
\begin{align}
\phi_{p}(N) =
    \begin{cases} \displaystyle
          \begin{aligned} & \frac{(2p)! 2^{\frac{N}{2}} } {\Gamma(\frac{N}{2}) }
                        \sum_{j=0}^{p-\frac{N}{2}} \sum_{i=0}^{\frac{N}{2}-1}{N-1 \choose 2j}\frac{(-1)^{j}2^{-j-2i}}{(2j+2i+1)!(p-\frac{N}{2}-j)!} \\
                          & \qquad	+ \frac{(2p)!}{\Gamma(\frac{N}{2})}
                        \sum_{j=0}^{\frac{N}{2}-1} \sum_{i=0}^{j}\frac{(\frac{N}{2}-i-1)!}{(j-i)!(p-j)!2^{p-2j}} {N-1 \choose N-2i-1},
          \end{aligned}
          & N \le 2p ,
    \\[+40pt] \displaystyle
          (2p)! \sum_{j=0}^{p} \frac{\left( \frac{N}{2}+\frac{1}{2}-j \right)_{j}}{2^{-3j}(2j)!(p-j)!},
          & N > 2p .
    \end{cases}
\label{GOE_phi}
\end{align}
\end{theorem}

\begin{remark}
For $ 0\leq p\leq 6 $ this agrees with Eqs. \eqref{moment0}-\eqref{moment12} in the case $ \kappa = 1/2 $.
\end{remark}

\subsection{$\kappa =2$ GSE Moments}
We recount the well-known explicit result for the density of eigenvalues in the symplectic Gaussian ensemble,
as say given in \S 4, p. 159  of \cite{AFNvM_2000}, but adapted to our slightly differing conventions.
From this work we deduce
\begin{equation}
   \rho_{(1)}(x,N) = \frac{1}{2}\sqrt{\frac{N}{g}} \left\{
    \sqrt{2}K_{2N}(\sqrt{2}x,\sqrt{2}x) + \frac{e^{-\frac{1}{2}x^2}H_{2N}(x)}{\sqrt{\pi}2^{2N}(2N-1)!}\int^{x}_{-\infty}dt\;e^{-\frac{1}{2}t^2}H_{2N-1}(t)
                                                   \right\}_{x\mapsto \sqrt{\frac{N}{g}}x} ,
\label{GSE_density}
\end{equation}
where $ K_{N}(\cdot,\cdot) $ is given by the kernel \eqref{OPSkernel}.

All the results we give in this subsection for the GSE case can be expressed by the {\it duality relations} with
the GOE. In \eqref{full_duality} the general expression for all $ \kappa $ was given and also details of the
implications for the generating functions in Corollary \ref{dual}, so we will refrain from repeating all of that here.
\begin{theorem}[\cite{MW_2003}, Theorem 6 of \cite{Le_2009}]
For all $ p \in \mathbb{Z} $ and $ N\ge 1 $ the moments satisfy the duality relation
\begin{equation}
  m_{2p}(N,2) = (-)^{p+1}2^{-p-1} m_{2p}(-2N,1/2) ,
\label{duality}
\end{equation}
as implied by \eqref{full_duality} with $\kappa = 2$, together with the corresponding formula of Corollary \ref{dual}.
\end{theorem}
However, it is of independent interest to derive the results from first principles, and we will
proceed in this manner.

The symplectic analogue of \eqref{GUEresolventODE} and \eqref{GOEresolventODE} is a fifth order ordinary linear
differential equation for the density and an inhomogeneous version for the resolvent.
\begin{theorem}\label{GSEWode}
The resolvent $ W_{1}(x) $ satisfies the fifth order, linear inhomogeneous ordinary differential equation
\begin{multline}
  -\tfrac{1}{4}\frac{g^4}{N^4}W_{1}^{(V)}+5\left[ \tfrac{1}{4}x^2-\frac{g}{N}(N+\tfrac{1}{4}) \right]\frac{g^2}{N^2}W_{1}^{\prime\prime\prime}
    -\tfrac{3}{2}\frac{g^2}{N^2}xW_{1}^{\prime\prime}
\\
  +\left[ -x^4+(8N+2)\frac{g}{N}x^2+(-16N^2-8N+\tfrac{1}{2})\frac{g^2}{N^2} \right]W_{1}^{\prime}+x\left[ x^2-(4N+1)\frac{g}{N} \right]W_{1}
\\
  = 2 N (x^2-4 g)-5 g ,
\label{GSEresolventODE}
\end{multline}
subject to the boundary conditions \eqref{resolventBC}, for fixed $ g, N $, with \nonumber the moments given by \eqref{GSE_moments}.
\end{theorem}
\begin{proof}
As in Theorems \ref{GUEWode} and \ref{GOEWode} we will derive the result for the unscaled independent
variable and make the scaling at the conclusion of the derivation.
We take as our starting point a simplified variant of \eqref{GSE_density}
\begin{equation}
   B_N \rho_{(1)}(x) = e^{-x^2}\left[ H_{2N}^{\prime}(x)H_{2N-1}(x)-H_{2N}(x)H_{2N-1}^{\prime}(x) \right] + e^{-\frac{1}{2}x^2}H_{2N}(x)Q_{2N-1}(x) ,
\end{equation}
where $ B_N = \pi^{1/2}2^{2N+1}(2N-1)! $ and the new variable is defined as
\begin{equation}
   Q_{n}(x) := \int^{x}_{-\infty} dt\; e^{-\frac{1}{2}t^2} H_{n}(t) .
\end{equation}
We will work with the Hermite polynomials instead of the $ p_n $ to avoid unnecessary factors of two appearing in the
workings, and the identities we require that correspond to the ones employed for the $ p_n $ are
\begin{gather*}
   H_{n+1} = 2xH_{n}-2nH_{n-1} ,
\\
   H_{n}^{\prime} = 2nH_{n-1} ,
\\
   H_{n}^{\prime\prime}-2xH_{n}^{\prime}+2nH_{n} = 0 ,
\\
   Q_{n}^{\prime} = e^{-\frac{1}{2}x^2}H_{n} ,
\\
   H_{n}^{\prime}H_{n-1}^{\prime} + 2nH_{n}H_{n-1} - 2xH_{n}^{\prime}H_{n-1} = 0 .
\end{gather*}
Again we successively differentiate the density and employ the above identities to reduce the expressions
to linear combinations of the independent bilinear products
\begin{equation}
   \{ H_{2N}H_{2N-1}, H_{2N}^{\prime}H_{2N-1}, H_{2N}H_{2N-1}^{\prime}, H_{2N}Q_{2N-1}, H_{2N}^{\prime}Q_{2N-1} \} .
\end{equation}
The result for the first four derivatives is
\begin{multline}
  B_{N} \begin{pmatrix}
          \rho_{(1)} \\ \rho_{(1)}^{\prime} \\ \rho_{(1)}^{\prime\prime} \\ \rho_{(1)}^{\prime\prime\prime} \\ \rho_{(1)}^{IV}
        \end{pmatrix} =
\\
     \begin{pmatrix}
      0 & 1 & -1 & 1 & 0 \\
      -1 & 0 & 0 & -x & 1 \\
      x & 0 & -1 & x^2-4N-1 & 0 \\
      -x^2+4N-2 & -x & x & -x^3+(4N+3)x & x^2-4N-1 \\
      x^3-(4N-7)x & -4 & -x^2+4N-1 & x^4-(8N+6)x^2+16N^2+8N+3 & 4x
     \end{pmatrix}
\\ \cdot
     \begin{pmatrix}
      e^{-x^2}H_{2N}H_{2N-1} \\ e^{-x^2}H_{2N}^{\prime}H_{2N-1} \\ e^{-x^2}H_{2N}H_{2N-1}^{\prime} \\ e^{-\frac{1}{2}x^2}H_{2N}Q_{2N-1} \\ e^{-\frac{1}{2}x^2}H_{2N}^{\prime}Q_{2N-1}
     \end{pmatrix} .
\end{multline}
In this case the determinant of the transformation is $ -36(2N+1) $.
The fifth order derivative is computed to be
\begin{multline}
  B_{N}\rho_{(1)}^{V}
  = \left[ -x^4+(8N-19)x^2-16N^2+52N+8 \right]e^{-x^2}H_{2N}H_{2N-1}
\\
   + \left[ -x^3+(4N+1)x \right]e^{-x^2}H_{2N}^{\prime}H_{2N-1} + \left[ x^3+(-4N+5)x \right]e^{-x^2}H_{2N}H_{2N-1}^{\prime}
\\
   + \left[ -x^5+(8N+10)x^3-(16N^2+40N+15)x \right]e^{-\frac{1}{2}x^2}H_{2N}Q_{2N-1}
\\
   + \left[ x^4-(8N+2)x^2+16N^2+8N+7 \right]e^{-\frac{1}{2}x^2}H_{2N}^{\prime}Q_{2N-1} .
\end{multline}
and after substituting for the basis products we find the homogeneous fifth order ordinary differential equation
\begin{multline}
   -\tfrac{1}{4}\rho_{(1)}^{(V)}+5\left[ \tfrac{1}{4}x^2-(N+\tfrac{1}{4}) \right]\rho_{(1)}^{\prime\prime\prime}-\tfrac{3}{2}x\rho_{(1)}^{\prime\prime}
\\
  +\left[ -x^4+(8N+2)x^2-16N^2-8N+\tfrac{1}{2} \right]\rho_{(1)}^{\prime}+x\left[ x^2-(4N+1) \right]\rho_{(1)} = 0 .
\end{multline}
Restoring the bulk scaling $ x\mapsto \sqrt{\frac{N}{g}}x $ we have the homogeneous form of \eqref{GSEresolventODE}.
To find the differential equation for $ W_1 $ we repeat the methods employed in the proof of Theorems
\ref{GUEWode} and \ref{GOEWode}. Integrating the homogeneous form of the differential
equation against $ (z-x)^{-1} $ on $ \mathbb{R} $, and performing the subtractions for the $x$-dependent
coefficients we arrive at
\begin{multline}
 0 =  -\tfrac{1}{4}\frac{g^4}{N^4}\int^{\infty}_{-\infty}dx\;\frac{\rho_{(1)}^{(V)}}{(z-x)}
       +5\left[ \tfrac{1}{4}z^2-(N+\tfrac{1}{4})\frac{g}{N} \right]\frac{g^2}{N^2}\int^{\infty}_{-\infty}dx\;\frac{\rho_{(1)}^{\prime\prime\prime}}{(z-x)}
\\
     -\tfrac{3}{2}\frac{g^2}{N^2}z\int^{\infty}_{-\infty}dx\;\frac{\rho_{(1)}^{\prime\prime}}{(z-x)}
      +\left[ -z^4+(8N+2)\frac{g}{N}z^2+(-16N^2-8N+\tfrac{1}{2})\frac{g^2}{N^2} \right]\int^{\infty}_{-\infty}dx\;\frac{\rho_{(1)}^{\prime}}{(z-x)}
\\
       +z\left[ z^2-(4N+1)\frac{g}{N} \right]\int^{\infty}_{-\infty}dx\;\frac{\rho_{(1)}}{(z-x)}
\\
      -\tfrac{5}{4}\frac{g^2}{N^2}\int^{\infty}_{-\infty}dx\;(x+z)\rho_{(1)}^{\prime\prime\prime}
       -\tfrac{3}{2}\frac{g^2}{N^2}\int^{\infty}_{-\infty}dx\;\rho_{(1)}^{\prime\prime}
        +\int^{\infty}_{-\infty}dx\;[z^3+z^2x+zx^2+x^3-(8N+2)\frac{g}{N}(x+z)]\rho_{(1)}^{\prime}
\\
         +\int^{\infty}_{-\infty}dx\;[-z^2-zx-x^2+(4N+1)\frac{g}{N}]\rho_{(1)} .
\end{multline}
Making similar observations on $ \rho_{(1)} $ concerning its decay as $ x\to \pm\infty $ as we did in the
proof of Theorems \ref{GUEWode}, \ref{GOEWode} we can use
$ \int dx\, (z-x)^{-1}\partial^n_x\rho_{(1)}(x) = \partial^n_z W_1(z) $ for $ 0 \leq n \leq 5 $. Only a few of the
remaining terms are non-zero and these contribute
$ [ -2z^2+(12N+3)\frac{\displaystyle g}{N}]m_{0}-4\frac{\displaystyle g}{N}m_{2} $.
Using $ m_{0}=N $ and $ m_{2}=N(N-\frac{1}{2}) $ we deduce the inhomogeneous term and arrive at \ref{GSEresolventODE}.
\end{proof}
\begin{remark}
As a check we can substitute the $ W_1 $ duality formula \eqref{Dual_W} into GSE ordinary differential equation
\eqref{GSEresolventODE} and easily recover its GOE equivalent \eqref{GOEresolventODE}.
Furthermore we can verify that the $ 1/N $ expansion of the resolvent \eqref{resolventN} along with
coefficients \eqref{resolvent_0}-\eqref{resolvent_6}, under the specialisation $ \kappa \to 2 $, identically satisfies
\eqref{GSEresolventODE} up to the error term of $ {\rm O}(N^{-7}) $.
\end{remark}

The result for the linear ordinary differential equation for the exponential generating function
in the GOE case Theorem \ref{GOEu} has a symplectic analogue. This was also given by Ledoux \cite{Le_2009} but
can be directly deduced from the previous result.
In the symplectic case Ledoux has defined the moments $ c^{N}_{p} = 2^{p} m_{2p}(N,2) $.
\begin{theorem}
The GSE exponential generating function $ u(t,N) $ satisfies the fourth order, linear differential equation
\begin{multline}
t u^{{\rm (IV)}} + 5 u^{\prime\prime\prime} - t( \tfrac{5}{4}t^2+8N+2 ) u^{\prime\prime} - ( 9t^2+20N+5 ) u^{\prime}
\\
+ t \left[ \tfrac{1}{4}t^4+5(N+\tfrac{1}{4})t^2+16N^2+8N-11 \right] u =0 ,
\label{GSE_uODE}
\end{multline}
subject to the boundary conditions
\begin{equation}
    u(t,N) \mathop{\sim}_{t \rightarrow 0 } N + \frac{1}{2!}N(N-\tfrac{1}{2}) t^{2} + \frac{1}{4!}N(2N^2-\tfrac{5}{2}N+\tfrac{5}{4}) t^{4}
                                              + \frac{1}{6!}N(5N^3-11N^2+13N-\tfrac{41}{8}) t^{6} \cdots .
\label{GSE_bcODE}
\end{equation}
\end{theorem}
\begin{proof}
The method is the same as in the case of the GUE and GOE cases so we confine ourselves to recording
the intermediate step
\begin{multline}
 0 = \frac{g}{N} \int^{\infty}_{0}dt\, \left\{ \partial^4_{t}(tu)+\partial^3_{t}u-\partial^2_{t}[\tfrac{5}{4}t^3u+(8N+2)tu]-\partial_{t}[\tfrac{3}{2}t^2u+(4N+1)u]
     \right.
\\   \left.
     +[\tfrac{1}{4}t^5+5(N+\tfrac{1}{4})t^3-(-16N^2-8N+\tfrac{1}{2})t]u \right\}
\\
     + \frac{g}{N} \left[ e^{-\sqrt{\frac{N}{g}}xt}\left( -\partial^3_{t}(tu)-\partial^2_{t}u+\partial_{t}((\tfrac{5}{4}t^2+8N+2)tu)+(\tfrac{3}{2}t^2+4N+1)u \right)
                   \right.
\\                 \left.
                      + \partial_{t}e^{-\sqrt{\frac{N}{g}}xt} \left( \partial^2_{t}(tu)+\partial_{t}u-(\tfrac{5}{4}t^2+8N+2)tu \right)
                      + \partial^2_{t}e^{-\sqrt{\frac{N}{g}}xt}\left( -\partial_{t}(tu)-u \right) + \partial^3_{t}e^{-\sqrt{\frac{N}{g}}xt}(tu)
                   \right]^{\infty}_{0}
\\   -2N(x^2-4g)+5 g.
\end{multline}
Again assuming the upper terminal contribution of the boundary term vanish we compute the lower terminal to be
$ \frac{\displaystyle g}{\displaystyle N}\left(4u^{\prime\prime}(0)-(12N+3)u(0)+2\frac{\displaystyle N}{\displaystyle g}x^2u(0) \right) $.
Using the data for $ m_0, m_2 $ this cancels the inhomogeneous terms from the original differential equation and
we have \eqref{GSE_uODE}.
\end{proof}
\begin{remark}
Employing the duality formula for $ u $, \eqref{Dual_u}, and the implied mapping of the independent variable
into the preceding symplectic formula \eqref{GSE_uODE} we recover \eqref{GOE_uODE}.
\end{remark}

Ledoux has shown that a linear recurrence relation for the moments follows from the above result. This also follows
directly from the ordinary differential equation for $ W_1 $ as given in Theorem \ref{GSEWode}.
\begin{theorem}[\cite{Le_2009}]
The GSE moments $ m_{2p}(N,2) $ satisfy the fourth order, linear difference equation
\begin{multline}
  (p+1)m_{2p} = \tfrac{1}{2}(4p-1)(4N+1)m_{2p-2}+\tfrac{1}{4}(2p-3)\left(10p^2-9p-32N^2-16N\right)m_{2p-4}
\\
              - \tfrac{5}{8}(2p-3)(2p-4)(2p-5)(4N+1)m_{2p-6} - \tfrac{1}{8}(2p-3)(2p-4)(2p-5)(2p-6)(2p-7)m_{2p-8} ,
\label{GSE_RR}
\end{multline}
subject to the initial values $ m_{0}=N $, $ m_{2}=N(N-\tfrac{1}{2}) $ for $ p\geq 2 $.
\end{theorem}
\begin{proof}
As in the proof of Theorems \ref{GUErecurrence}, \ref{GOErecurrence} we integrate $ x^{2p-3} $ against the homogeneous form of
\eqref{GSEresolventODE} for $ \rho_{(1)} $, and after integrating by parts we find
\begin{multline}
 0 = \left(\frac{g}{N}\right)^p
     \Big\{
     (2p+2)m_{2p}-(4N+1)(4p-1)m_{2p-2}
\\
     +(2p-3)[16N^2+8N-\tfrac{1}{2}-\tfrac{3}{2}(2p-2)-\tfrac{5}{4}(2p-1)(2p-2)]m_{2p-4}
\\
      +\tfrac{5}{4}(4N+1)(2p-3)(2p-4)(2p-5)m_{2p-6}
\\
      +\tfrac{1}{4}(2p-3)(2p-4)(2p-5)(2p-6)(2p-7)m_{2p-8} \Big\}
\\
      + \text{boundary terms containing $ x^{2p+1}\rho_{(1)}, \ldots ,x^{2p-3}\rho_{(1)}^{(IV)} $ as $ x \to \pm\infty $} .
\end{multline}
Again $ x^{2p+1}\rho_{(1)} $, $ x^{2p-2}\rho_{(1)}^{\prime} $, $ x^{2p-1}\rho_{(1)}^{\prime\prime} $, $ x^{2p-4}\rho_{(1)}^{\prime\prime\prime} $,
$ x^{2p-3}\rho_{(1)}^{(IV)} $ all vanish exponentially fast as $ x \to \pm\infty $ for all $ p\geq 2 $ and we are justified in
neglecting the boundary terms. Eq. \eqref{GSE_RR} then follows.
\end{proof}

\begin{remark}
Employing the duality formula \eqref{duality} into the GSE recurrence \eqref{GSE_RR} we recover the GOE
analog \eqref{GOE_RR}.
\end{remark}

Initial data on the GSE moments are efficiently computed with this recurrence and we give the first seven cases
\begin{equation}
\begin{aligned}\label{GSE_moments}
m_{0} & = N ,
\\
2m_{2} & = 2N^{2} - N ,
\\
4m_{4} & = 8N^{3} - 10N^{2} + 5N ,
\\
8m_{6} & = 40N^{4} - 88N^{3} + 104N^{2} - 41N ,
\\
16m_{8} & = 224N^{5} - 744N^{4} +1496N^{3} - 1380N^{2} + 509N ,
\\
32m_{10} & = 1344N^{6} - 6176N^{5} + 18320N^{4} - 28600N^{3} + 24286N^{2} - 8229N ,
\\
64m_{12} & = 8448 N^{7} - 50752 N^{6} + 204768 N^{5} - 470080 N^{4} + 668592 N^{3} - 516958 N^{2} + 166377 N .
\end{aligned}
\end{equation}

In addition we can make statements about the leading terms of the general $ 2p$-th moment, in the sense of large
$ N $ using the recurrence relation.
\begin{theorem}\label{GSEfirst}
As $N \rightarrow \infty$ with $ p $ fixed then $ m_{2p} = O(N^{p+1}) $ and the leading coefficients are given by
\begin{multline}
  \frac{m_{2p}}{N^{p+1}} = C_{p} + 4^{p-1} \left[ -1+\frac{\Gamma(p+\frac{1}{2})}{\sqrt{\pi}\Gamma(p+1)} \right]N^{-1}
\\
    + \frac{1}{3} 4^{p-2}p \left[ -3+\frac{(7p-1)\Gamma(p+\frac{1}{2})}{\sqrt{\pi}\Gamma(p+1)} \right]N^{-2}
\\
    + \frac{1}{3}2^{2p-7}p(p-1) \left[ -8p+7+\frac{2(7p-2)\Gamma(p+\frac{1}{2})}{\sqrt{\pi}\Gamma(p+1)} \right]N^{-3}
\\
    + \frac{1}{45}2^{2p-9}p(p-1)(p-2)\left[ -15(8p-9) +\frac{(185p^2-317p+6)\Gamma(p+\frac{1}{2})}{\sqrt{\pi}\Gamma(p+1)} \right]N^{-4}
\\
    + \frac{1}{45}2^{2p-13}p(p-1)(p-2)(p-3) \left[ -320p^2+1008p-487+\frac{4(185p^2-387p+28)\Gamma(p+\frac{1}{2})}{\sqrt{\pi}\Gamma(p+1)} \right]N^{-5}
\\
    + \frac{1}{2835}2^{2p-15}p(p-1)(p-2)(p-3)(p-4) \left[ -63(320p^2-1168p+675) \phantom{\frac{\Gamma(p+\frac{1}{2})}{\sqrt{\pi}\Gamma(p+1)}} \right.
\\ \left.
    +\frac{4(6209p^3-29106p^2+26605p-60)\Gamma(p+\frac{1}{2})}{\sqrt{\pi}\Gamma(p+1)} \right]N^{-6}
    + {\rm O}(N^{-7}) .
\label{GSE_expN}
\end{multline}
\end{theorem}
\begin{proof}
This result can be shown using the methods of Theorem \ref{GUE_largeNmoment}, with the additional
features noted in proof of Theorem \ref{GOEfirst}.
\end{proof}

\begin{remark}
This agrees with Theorem \ref{T19} in the case $ \kappa = 2 $.
\end{remark}

Again one has an explicit formula for the moments, given by Mezzadri and Simm 2011 \cite{MS_2011}, Equation (33)
where we note $ m_{2p}(N,2) \equiv 2^{p}M_{G}^{(4)}(2p,N) $.
\begin{theorem}[\cite{MS_2011}]
For all $ N>0 $ the GSE moments are given by
\begin{multline}
m_{2p}(N,2) = 2^{-p-1} m_{2p}(2N,1) \\
                  - \frac{\Gamma(N+1)\Gamma(N)}{\pi^{1/2} 4^{1-N} \Gamma(2N)}
                    \sum_{j=1}^{\min(N,p)} \sum_{i=0}^{\min(N-j,p-j)}
                    {p \choose i} {p \choose i+j} (N-i-j+1)_{p-\frac{1}{2}} .
\label{GSE_finitesum}
\end{multline}
\end{theorem}

\begin{remark}
For $ 0\leq p\leq 6 $ this agrees with Eqs. \eqref{moment0}-\eqref{moment12} in the case $ \kappa = 2 $.
\end{remark}

\section{Acknowledgements}
This work was supported by the Australian Research Council through the DP `Characteristic polynomials in random matrix theory'.

\section{Appendix. Additional Data}
\setcounter{equation}{0}

In this Appendix we record the values of $ \tilde{\rho}_{(1),4}(x) $, $ \tilde{\rho}_{(1),5}(x) $ and $ \tilde{\rho}_{(1),6}(x) $
computed as for $ \tilde{\rho}_{(1),0}(x) $, $ \dots $, $ \tilde{\rho}_{(1),3}(x) $ in (\ref{density0})--(\ref{density3}). We find
{\small \begin{multline}
  \tilde{\rho}_{(1),4}(x) =
   h^4 \left\{ -\frac{1}{\pi}(37x^4+123gx^2+21g^2)(4g-x^2)^{-11/2}\;\chi_{x \in (-2\sqrt{g},2\sqrt{g})}
       \right.
       \\ \left.
               -\frac{1}{2048g^{5/2}}\epsilon^{(1)}_{(-2\sqrt{g},2\sqrt{g})}
                -\frac{1}{1024g^2}\epsilon^{(2)}_{(-2\sqrt{g},2\sqrt{g})}
               -\frac{1}{96g^{3/2}}\epsilon^{(3)}_{(-2\sqrt{g},2\sqrt{g})}
                -\frac{15}{768g}\epsilon^{(4)}_{(-2\sqrt{g},2\sqrt{g})}
          \right\}
   \\
     +
   h^2 \left\{ -\frac{1}{2\pi}(23x^4+454gx^2+176g^2)(4g-x^2)^{-11/2}\;\chi_{x \in (-2\sqrt{g},2\sqrt{g})}
       \right.
       \\ \left.
               -\frac{39}{4096g^{5/2}}\epsilon^{(1)}_{(-2\sqrt{g},2\sqrt{g})}
                -\frac{39}{2048g^2}\epsilon^{(2)}_{(-2\sqrt{g},2\sqrt{g})}
               -\frac{7}{384g^{3/2}}\epsilon^{(3)}_{(-2\sqrt{g},2\sqrt{g})}
                -\frac{17}{1536g}\epsilon^{(4)}_{(-2\sqrt{g},2\sqrt{g})}
          \right\}
   \\
     -\frac{1}{\pi}21g(x^2+g)(4g-x^2)^{-11/2}\;\chi_{x \in (-2\sqrt{g},2\sqrt{g})},
\label{density4}
\end{multline}{
{\small \begin{multline}
  \tilde{\rho}_{(1),5}(x) =
   h^5 \left\{  \frac{1}{\pi}(353x^4+1527gx^2+399g^2)(4g-x^2)^{-13/2}\;\chi_{x \in (-2\sqrt{g},2\sqrt{g})}
       \right.
       \\
               +\frac{425}{524288 g^{7/2}}\epsilon^{(1)}_{(-2\sqrt{g},2\sqrt{g})}
               +\frac{425}{262144 g^3}\epsilon^{(2)}_{(-2\sqrt{g},2\sqrt{g})}
               +\frac{159}{49152 g^{5/2}}\epsilon^{(3)}_{(-2\sqrt{g},2\sqrt{g})}
       \\ \left.
               +\frac{847}{196608 g^2}\epsilon^{(4)}_{(-2\sqrt{g},2\sqrt{g})}
               +\frac{705}{491520 g^{3/2}}\epsilon^{(5)}_{(-2\sqrt{g},2\sqrt{g})}
               -\frac{1695}{737280 g}\epsilon^{(6)}_{(-2\sqrt{g},2\sqrt{g})}
          \right\}
   \\
     +
   h^3 \left\{  \frac{1}{2\pi}(445x^4+4332gx^2+1512g^2)(4g-x^2)^{-13/2}\;\chi_{x \in (-2\sqrt{g},2\sqrt{g})}
       \right.
       \\
               +\frac{3019}{1048576 g^{7/2}}\epsilon^{(1)}_{(-2\sqrt{g},2\sqrt{g})}
               +\frac{3019}{524288 g^3}\epsilon^{(2)}_{(-2\sqrt{g},2\sqrt{g})}
               +\frac{157}{98304 g^{5/2}}\epsilon^{(3)}_{(-2\sqrt{g},2\sqrt{g})}
       \\ \left.
               -\frac{1763}{393216 g^2}\epsilon^{(4)}_{(-2\sqrt{g},2\sqrt{g})}
               -\frac{5837}{983040 g^{3/2}}\epsilon^{(5)}_{(-2\sqrt{g},2\sqrt{g})}
               -\frac{5677}{1474560 g}\epsilon^{(6)}_{(-2\sqrt{g},2\sqrt{g})}
          \right\}
   \\
     +
   h   \left\{  \frac{1}{2\pi}(21x^4+420gx^2+294g^2)(4g-x^2)^{-13/2}\;\chi_{x \in (-2\sqrt{g},2\sqrt{g})}
       \right.
       \\
               -\frac{1533}{524288 g^{7/2}}\epsilon^{(1)}_{(-2\sqrt{g},2\sqrt{g})}
               -\frac{1533}{262144 g^3}\epsilon^{(2)}_{(-2\sqrt{g},2\sqrt{g})}
               -\frac{327}{49152 g^{5/2}}\epsilon^{(3)}_{(-2\sqrt{g},2\sqrt{g})}
       \\ \left.
               -\frac{1083}{196608 g^2}\epsilon^{(4)}_{(-2\sqrt{g},2\sqrt{g})}
               -\frac{1533}{491520 g^{3/2} }\epsilon^{(5)}_{(-2\sqrt{g},2\sqrt{g})}
               -\frac{717}{737280 g}\epsilon^{(6)}_{(-2\sqrt{g},2\sqrt{g})}
          \right\} ,
\label{density5}
\end{multline}}
\begin{multline}
  \tilde{\rho}_{(1),6}(x) =
   h^6 \left\{  \frac{1}{\pi}(4081x^6+28625g x^4+26832g^2 x^2+1738g^3)(4g-x^2)^{-17/2}\;\chi_{x \in (-2\sqrt{g},2\sqrt{g})}
       \right.
       \\
               +\frac{161}{2097152 g^{9/2}}\epsilon^{(1)}_{(-2\sqrt{g},2\sqrt{g})}
               +\frac{161}{1048576 g^4}\epsilon^{(2)}_{(-2\sqrt{g},2\sqrt{g})}
               +\frac{1197}{1572864 g^{7/2}}\epsilon^{(3)}_{(-2\sqrt{g},2\sqrt{g})}
       \\ \left.
               +\frac{259}{196608 g^3}\epsilon^{(4)}_{(-2\sqrt{g},2\sqrt{g})}
               +\frac{1849}{983040 g^{5/2}}\epsilon^{(5)}_{(-2\sqrt{g},2\sqrt{g})}
               +\frac{6075}{2949120g^2}\epsilon^{(6)}_{(-2\sqrt{g},2\sqrt{g})}
               +\frac{11865}{10321920 g^{3/2}}\epsilon^{(7)}_{(-2\sqrt{g},2\sqrt{g})}
          \right\}
   \\
     +
   h^4 \left\{  \frac{1}{2\pi}\left(8567x^6+147556g x^4+243180g^2 x^2+31236g^3\right)(4g-x^2)^{-17/2}\;\chi_{x \in (-2\sqrt{g},2\sqrt{g})}
       \right.
       \\
               +\frac{7987}{4194304 g^{9/2}}\epsilon^{(1)}_{(-2\sqrt{g},2\sqrt{g})}
               +\frac{7987}{2097152 g^4}\epsilon^{(2)}_{(-2\sqrt{g},2\sqrt{g})}
               +\frac{27543}{3145728 g^{7/2}}\epsilon^{(3)}_{(-2\sqrt{g},2\sqrt{g})}
       \\ \left.
               +\frac{4889}{393216 g^3}\epsilon^{(4)}_{(-2\sqrt{g},2\sqrt{g})}
               +\frac{20683}{1966080 g^{5/2}}\epsilon^{(5)}_{(-2\sqrt{g},2\sqrt{g})}
               +\frac{34305}{5898240 g^2}\epsilon^{(6)}_{(-2\sqrt{g},2\sqrt{g})}
               +\frac{39739}{20643840 g^{3/2}}\epsilon^{(7)}_{(-2\sqrt{g},2\sqrt{g})}
          \right\}
   \\
     +
   h^2 \left\{  \frac{1}{\pi}\left(618x^6+32043g x^4+91299g^2 x^2+16834g^3\right)(4g-x^2)^{-17/2}\;\chi_{x \in (-2\sqrt{g},2\sqrt{g})}
       \right.
       \\
               +\frac{10731}{2097152 g^{9/2}}\epsilon^{(1)}_{(-2\sqrt{g},2\sqrt{g})}
               +\frac{10731}{1048576 g^4}\epsilon^{(2)}_{(-2\sqrt{g},2\sqrt{g})}
               +\frac{17679}{1572864 g^{7/2}}\epsilon^{(3)}_{(-2\sqrt{g},2\sqrt{g})}
       \\ \left.
               +\frac{1737}{196608 g^3}\epsilon^{(4)}_{(-2\sqrt{g},2\sqrt{g})}
               +\frac{5007}{983040 g^{5/2}}\epsilon^{(5)}_{(-2\sqrt{g},2\sqrt{g})}
               +\frac{6033}{2949120 g^2}\epsilon^{(6)}_{(-2\sqrt{g},2\sqrt{g})}
               +\frac{5019}{10321920 g^{3/2}}\epsilon^{(7)}_{(-2\sqrt{g},2\sqrt{g})}
          \right\}
   \\
     +\frac{1}{\pi}(1485g x^4+6138 g^2 x^2+1738 g^3)(4g-x^2)^{-17/2}\;\chi_{x \in (-2\sqrt{g},2\sqrt{g})} .
\label{density6}
\end{multline}
Furthermore we record four higher moments calculated using the MOPS software package \cite{DES_2007} instead
of the methods employed to compute \eqref{moment0}-\eqref{moment12}
\begin{align}
m_{14} = & 
429 N^8-6476 N^7 (-\kappa^{-1}+1)+28 N^6 \left(1550\kappa^{-2}-2671\kappa^{-1}+1550\right)
\label{moment14} \\
         &  -14 N^5 \left(-11865\kappa^{-3}+26521\kappa^{-2}-26521\kappa^{-1}+11865\right)
\notag\\
         &  +7 N^4 \left(55448\kappa^{-4}-143753\kappa^{-3}+186048\kappa^{-2}-143753\kappa^{-1}+55448\right)
\notag\\
         &  -14 N^3 \left(-39034\kappa^{-5}+110855\kappa^{-4}-165733\kappa^{-3}+165733\kappa^{-2}-110855\kappa^{-1}+39034\right)
\notag\\
         &  +N^2 \left(422232\kappa^{-6}-1270913\kappa^{-5}+2070257\kappa^{-4}-2386524\kappa^{-3}+2070257\kappa^{-2}-1270913\kappa^{-1}+422232\right)
\notag\\
         &  +N \left(135135\kappa^{-7}-422232\kappa^{-6}+724437\kappa^{-5}-906423\kappa^{-4}+906423\kappa^{-3}-724437\kappa^{-2}+422232\kappa^{-1}-135135\right) ,
\notag
\end{align}
\begin{align}
m_{16} = & 
1430 N^9-26333 N^8 (-\kappa^{-1}+1)+4 N^7 \left(55177\kappa^{-2}-95339\kappa^{-1}+55177\right)
\label{moment16} \\
         &  -14 N^6 \left(-78040\kappa^{-3}+175407\kappa^{-2}-175407\kappa^{-1}+78040\right)
\notag\\
         &  +N^5 \left(3463634\kappa^{-4}-9056368\kappa^{-3}+11756038\kappa^{-2}-9056368\kappa^{-1}+3463634\right)
\notag\\
         &  +N^4 \left(7123780\kappa^{-5}-20466843\kappa^{-4}+30790276\kappa^{-3}-30790276\kappa^{-2}+20466843\kappa^{-1}-7123780\right)
\notag\\
         &  +N^3 \left(9163236\kappa^{-6}-27995000\kappa^{-5}+46050702\kappa^{-4}-53268136\kappa^{-3} \right. \notag\\
         & \phantom{+12 N^3 (}\left. +46050702\kappa^{-2}-27995000\kappa^{-1}+9163236\right)
\notag\\
         &  +N^2 \left(6633360\kappa^{-7}-21117210\kappa^{-6}+36735448\kappa^{-5}-46305896\kappa^{-4} \right. \notag\\
         & \phantom{+12 N^3 (}\left. +46305896\kappa^{-3}-36735448\kappa^{-2}+21117210\kappa^{-1}-6633360\right)
\notag\\
         &  +3 N \left(675675\kappa^{-8}-2211120\kappa^{-7}+3984658\kappa^{-6}-5288076\kappa^{-5}+5752801\kappa^{-4} \right. \notag\\
         & \phantom{+12 N^3 (}\left. -5288076\kappa^{-3}+3984658\kappa^{-2}-2211120\kappa^{-1}+675675\right) ,
\notag
\end{align}
\begin{align}
m_{18} = &
4862 N^{10}-106762 N^9 (-\kappa^{-1}+1)+6 N^8 \left(181261\kappa^{-2}-313902\kappa^{-1}+181261\right)
\label{moment18} \\
         &  -60 N^7 \left(-111789\kappa^{-3}+252415\kappa^{-2}-252415\kappa^{-1}+111789\right)
\notag\\
         &  +N^6 \left(27391174\kappa^{-4}-72116946\kappa^{-3}+93841930\kappa^{-2}-72116946\kappa^{-1}+27391174\right)
\notag\\
         &  -6 N^5 \left(-12684669\kappa^{-5}+36783020\kappa^{-4}-55611546\kappa^{-3}+55611546\kappa^{-2}-36783020\kappa^{-1}+12684669\right)
\notag\\
         &  +N^4 \left(142341934\kappa^{-6}-439988319\kappa^{-5}+729284620\kappa^{-4} \right. \notag\\
         & \phantom{+12 N^3 (}\left. -845821890\kappa^{-3}+729284620\kappa^{-2}-439988319\kappa^{-1}+142341934\right)
\notag\\
         &  -10 N^3 \left(-17063718\kappa^{-7}+55103324\kappa^{-6}-96859509\kappa^{-5}+122769969\kappa^{-4} \right. \notag\\
         & \phantom{+12 N^3 (}\left. -122769969\kappa^{-3}+96859509\kappa^{-2}-55103324\kappa^{-1}+17063718\right)
\notag\\
         &  +N^2 \left(117193185\kappa^{-8}-390187530\kappa^{-7}+712745500\kappa^{-6}-954191664\kappa^{-5} \right. \notag\\
         & \phantom{+12 N^3 (}\left. +1041198895\kappa^{-4}-954191664\kappa^{-3}+712745500\kappa^{-2}-390187530\kappa^{-1}+117193185\right)
\notag\\
         &  -3 N \left(-11486475\kappa^{-9}+39064395\kappa^{-8}-73183450\kappa^{-7}+101351398\kappa^{-6}-116492293\kappa^{-5} \right. \notag\\
         & \phantom{+12 N^3 (}\left. +116492293\kappa^{-4}-101351398\kappa^{-3}+73183450\kappa^{-2}-39064395\kappa^{-1}+11486475\right) ,
\notag
\end{align}
\begin{align}
m_{20} = & 16796 N^{11}-431910 N^{10} (-\kappa^{-1}+1)+10 N^9 \left(523069\kappa^{-2}-907571\kappa^{-1}+523069\right)
\label{moment20} \\
         &  -15 N^8 \left(-2605750\kappa^{-3}+5906423\kappa^{-2}-5906423\kappa^{-1}+2605750\right)
\notag\\
         &  +8 N^7 \left(24778268\kappa^{-4}-65615565\kappa^{-3}+85554470\kappa^{-2}-65615565\kappa^{-1}+24778268\right)
\notag\\
         &  -70 N^6 \left(-10102057\kappa^{-5}+29519110\kappa^{-4}-44811613\kappa^{-3}+44811613\kappa^{-2}-29519110\kappa^{-1}+10102057\right)
\notag\\
         &  +2 N^5 \left(890196239\kappa^{-6}-2777967945\kappa^{-5} \right. \notag\\
         & \phantom{+2 N^5 (}\left. +4632873326\kappa^{-4}-5384661375\kappa^{-3}+4632873326\kappa^{-2}-2777967945\kappa^{-1}+890196239\right)
\notag\\
         &  -5 N^4 \left(-618257450\kappa^{-7}+2019452031\kappa^{-6}-3579106742\kappa^{-5} \right. \notag\\
         & \phantom{-5 N^4 (}\left. +4556290742\kappa^{-4}-4556290742\kappa^{-3}+3579106742\kappa^{-2}-2019452031\kappa^{-1}+618257450\right)
\notag\\
         &  +2 N^3 \left(1750159371\kappa^{-8}-5906104210\kappa^{-7}+10901709075\kappa^{-6}-14692250235\kappa^{-5} \right. \notag\\
         & \phantom{+12 N^3 (}\left. +16068813521\kappa^{-4}-14692250235\kappa^{-3}+10901709075\kappa^{-2}-5906104210\kappa^{-1}+1750159371\right)
\notag\\
         &  -5 N^2 \left(-460192905\kappa^{-9}+1590096591\kappa^{-8}-3017610500\kappa^{-7}+4217705240\kappa^{-6}-4871156831\kappa^{-5} \right. \notag\\
         & \phantom{-5 N^2 (}\left. +4871156831\kappa^{-4}-4217705240\kappa^{-3}+3017610500\kappa^{-2}-1590096591\kappa^{-1}+460192905\right)
\notag\\
         &  +3 N \left(218243025\kappa^{-10}-766988175\kappa^{-9}+1483388071\kappa^{-8}-2122377110\kappa^{-7}+2533991909\kappa^{-6} \right. \notag\\
         & \phantom{+3 N (}\left. -2672675165\kappa^{-5}+2533991909\kappa^{-4}-2122377110\kappa^{-3}+1483388071\kappa^{-2}-766988175\kappa^{-1}+218243025\right) .
\notag
\end{align}

\bibliographystyle{plain}
\bibliography{moment,random_matrices,nonlinear,CA}

\end{document}